\documentclass[12pt]{amsart}
\usepackage{color}
\usepackage{axodraw4j}
\usepackage{pstricks}
\usepackage{amssymb}
\usepackage{mathtools}
\usepackage{txfonts}
\usepackage{eucal}
\usepackage{extarrows}
\usepackage{wasysym}
\usepackage[all]{xy}

\usepackage{caption}
\usepackage{float} 
\usepackage{amsmath}
\usepackage{tikz}
\usepackage{cite}
\usepackage{graphicx}
\usepackage{float}

\usepackage{xspace}
\usepackage[a4paper,body={16.3cm,22.8cm},centering]{geometry}
\usepackage[colorlinks=true,pagebackref=true]{hyperref} 
\hypersetup{urlcolor=blue, citecolor=red , linkcolor= blue}
\usepackage{shuffle}
\font \eightrm=cmr8

\newcommand{\nc}{\newcommand}
\nc\smsc{0.8}

\def\soprec{\,\joinrel{\ocircle\hskip -6.7pt \prec}\,}

\nc\delete[1]{}
\nc{\mlabel}[1]{\label{#1}}  
\nc{\mcite}[1]{\cite{#1}}  
\nc{\mref}[1]{\ref{#1}}  
\nc{\mbibitem}[1]{\bibitem{#1}} 

\delete{
\nc{\mlabel}[1]{\label{#1}  
{\hfill \hspace{1cm}{\small\tt{{\ }\hfill(#1)}}}}
\nc{\mcite}[1]{\cite{#1}{\small{\tt{{\ }(#1)}}}}  
\nc{\mref}[1]{\ref{#1}{{\tt{{\ }(#1)}}}}  
\nc{\mbibitem}[1]{\bibitem[\bf #1]{#1}} 
}
\delete{
\nc{\mop}[1]{\mathop{\hbox {\rm #1} }}
\nc{\smop}[1]{\mathop{\hbox {\eightrm #1} }}
\nc{\mopl}[1]{\mathop{\hbox {\rm #1} }\limits}
\nc{\smopl}[1]{\mathop{\hbox {\eightrm #1} }\limits}
\def \restr#1{\mathstrut_{\textstyle |}\raise-8pt\hbox{$\scriptstyle #1$}}
\def \srestr#1{\mathstrut_{\scriptstyle |}\hbox to
  -1.5pt{}\raise-4pt\hbox{$\scriptscriptstyle #1$}}
\nc{\wt}{\widetilde}
\nc{\wh}{\widehat}

}

\newtheorem{theorem}{Theorem}[section]
\newtheorem{definition}{Definition}[section]

\newtheorem{corollary}{Corollary}[section]
\newtheorem{proposition}{Proposition}[section]

\newtheorem{remark}{Remark}[section]
\newtheorem{example}{Example}[section]
\numberwithin{equation}{section}
\newcommand\alphlist{a,b,c,d,e,f,g,h,i,j,k,l,m,n,o,p,q,r,s,t,u,v,w,x,y,z}
\newcommand\Alphlist{A,B,C,D,E,F,G,H,I,J,K,L,M,N,O,P,Q,R,S,T,U,V,W,X,Y,Z}
\newcommand\getcmds[3]{\expandafter\newcommand\csname #2#1\endcsname{#3{#1}}}
\makeatletter
\@for\x:=\alphlist\do{\expandafter\getcmds\expandafter{\x}{frak}{\mathfrak}}
\@for\x:=\Alphlist\do{\expandafter\getcmds\expandafter{\x}{frak}{\mathfrak}}
\makeatother

\nc{\bfk}{{\bf k}}
\nc{\sha}{\shuffle}

\nc{\id}{\mathrm{id}}
\nc{\Id}{\mathrm{Id}}
\nc{\lbar}[1]{\overline{#1}}
\nc{\ot}{\otimes}
\nc{\dep}{\mathrm{dep}}
\nc{\ver}{\mathrm{ver}}

\nc{\tred}[1]{\textcolor{red}{#1}} \nc{\tgreen}[1]{\textcolor{green}{#1}}
\nc{\tblue}[1]{\textcolor{blue}{#1}} \nc{\tpurple}[1]{\textcolor{purple}{#1}}
\nc{\tcyan}[1]{\textcolor{cyan}{#1}} 
\nc{\tblk}[1]{\textcolor{black}{#1}}

\nc{\li}[1]{\tpurple{\underline{Li:}#1 }}
\nc{\liadd}[1]{\tpurple{#1}}
\nc{\xing}[1]{\tblue{\underline{Xing:}#1 }}
\nc{\yuan}[1]{\tred{\underline{Yuan:}#1 }}
\nc{\markus}[1]{\tred{\underline{Markus:} #1}}
\nc{\dominique}[1]{\tpurple{\underline{Dominique: }#1 }}
\long\def\ignore#1{}

\usetikzlibrary{calc,shapes.geometric}
\tikzset{
baseon/.style={baseline={($(#1)+(0,-0.58ex)$)}},
baseon/.default=current bounding box.center,
every picture/.style=baseon,
lst/.style={},
dst/.style={circle,inner sep=1pt,outer sep=0pt,fill,draw,dst2},
dst2/.style={fill=white},
ddst/.style={diamond,draw,inner sep=1pt},
eest/.style={ellipse,draw,inner sep=1pt,minimum size=2ex},
}

\def\zzz#1`#2...#3`#4...#5`#6@{%
--++(#1)
node[dst,label={#5:$#6$},name=#2]{}
node[midway,\hbox{Aut}o,#3]{$#4$}
}
\def\ddd#1`#2`#3@{+(#1)node[ddst,name=#2]{$#3$}}
\def\eee#1`#2`#3@{+(#1)node[eest,name=#2]{$#3$}}
\def\xxx#1`#2@{node[midway,\hbox{Aut}o,inner sep=1pt,#1]{$#2$}}
\def\pp#1`#2`#3@{node[dst,label={#2:$#3$},pos=#1]{}}
\def\oo#1`#2`#3@{\path (o) node[dst,label={#2:$#3$},name=o,#1]{};}
\def\eoo#1`#2@{\node[eest,name=o,#1] at (o) {$#2$};}

\newif\ifshowjdq
\showjdqtrue
\newcommand\setXXclip[3]{%
\def\XXheight{#1}\def\XXdepth{#2}\def\XXwidth{#3}}
\setXXclip{1}{-0.5}{1.1}

\makeatletter
\newcommand\simra{\mathrel{\mathpalette\@verra\sim}}
\def\@verra#1#2{\lower.5\p@\vbox{\lineskiplimit\maxdimen \lineskip-.5\p@
\ialign{$\m@th#1\hfil##\hfil$\crcr#2\crcr\rightarrow\crcr}}}
\makeatother

\nc{\dnx}{\Delta_n A} \nc{\dx}{\Delta A} \nc{\dgp}{{\rm deg_{P}}}
\nc{\dgt}{{\rm deg_{T}}} \nc{\dg}{{\rm deg}} \nc{\ida}{ID($A$)} \nc{\tu}{\tilde{u}} \nc{\tv}{\tilde{v}}
\nc{\nr}{\calr_n} \nc{\nz}{\calz_n} \nc{\fun}{\cala_{n,d}}
 \nc{\fbase}{\calb} \nc{\LF}{\mathrm{RF}} \nc{\FFA}{\mathrm{LF}} \nc{\irr}{\mathrm{Irr}}
 \nc{\result}{\bfk\mathrm{Irr}(S_n)}  \nc{\I}{I_{\mathrm{ID},n}^0}
 \nc{\nrs}{\calr_n^\star} \nc{\ii}{\mathrm{I}} \nc{\iii}{\mathrm{II}}
\nc{\intl}{{\rm int}}\nc{\ws}[1]{{#1}}\nc{\deleted}[1]{\delete{#1}}\nc{\plas}{placements\xspace}

\nc{\bim}[1]{#1}  \nc{\shaop}{\sha_{\Omega}^{+}}  \nc{\shao}{\sha_{\Omega}}
\nc{\bbim}[2]{#1 #2} \nc{\bbbim}[2]{#1,\, #2} \nc{\RBF}{{\rm RBF}}
\nc{\frb}{F_{\RB}} \nc{\shaf}{\ssha_{\tiny{\Omega}}} \nc{\sham}{\diamond_{\tiny{\Omega}}}
\nc{\lf}{\lfloor} \nc{\rf}{\rfloor} \nc{\shan}{\ssha_{\lambda}}
\nc{\rlex}{{\rm {lex}}} \nc{\bb}{\Box} \nc{\ra}{\rightarrow}
\nc{\e}{{\rm {e}}}
\nc{\DDF}{\mathrm{DD}(X,\,\Omega)}\nc{\DTF}{\mathrm{DT}(X,\,\Omega)} \nc{\DT}{\mathrm{DT}'(\Omega,\,V)}
\nc{\bra}{\mathrm{bra}} \nc{\bre}{\mathrm{bre}}
\nc{\dec}{\mathrm{dec}} \nc{\diamondw}{\diamond_{w}}
\nc{\type}{\mathrm{type}}

\nc\calt{\cal{T}(X,\,\Omega)} \nc\caltn{\cal{T}_n(X,\,\Omega)}
\nc\calta{\cal{T}_0(X,\,\Omega)}
\nc\caltb{\cal{T}_1(X,\,\Omega)}
\nc\caltc{\cal{T}_2(X,\,\Omega)}
\nc\caltd{\cal{T}_3(X,\,\Omega)}
\nc\caltm{\cal{T}_m(X,\,\Omega)}
\nc\caltx{\cal{T}(X)}
\nc\calf{\cal{F}(X,\,\Omega)}
\nc\fram{\frak{M}(\Omega,\, X)}
\nc\shaw{\sha^{NC}_w(\Omega,\, X)}
\nc\dw{\diamond_w} \nc\dl{\diamond_\ell}
\nc\shal{\sha^{NC}_\ell(X,\, \Omega)} \nc\shav{\sha^{NC}_w(\Omega,\, V)} \nc\shat{\sha^{NC,1}_w(\Omega,\, T^{+}(V))}
\nc{\cfo}{\cal{F}(X,\,\Omega)}
\nc{\sh}{\rm{Sh}}
\nc{\lar}{\varinjlim}
\def\cxo#1#2;{\cal{#1}#2\XO}
\nc\lrf[2]{B_{#2}^+(#1)}
\nc{\fd}{\mathrm{\text{typed angularly decorated planar rooted trees}}}
\nc{\rb}{\mathrm{RBFWs}} \nc{\dfw}{\mathrm{DFW{(X)}}} \nc{\tfw}{\mathrm{TFW{(X)}}}
\nc{\tfv}{\mathrm{TFW{(V)}}}

\def\Ve#1,#2,#3;{\vee_{#1,\,(#2,\,#3)}}
\def\bigv#1;#2;#3;{\bigvee\nolimits_{#1}^{#2;\,#3}}
\nc\rjt[2]{\mathrel{\mathop{\longrightarrow}\limits^{#1\hfill}_{\hfill#2}}}
\nc{\pl}{\cal{PLF}}
\nc{\tr}{\cal{RTF}}
\nc{\im}{\mathrm{Im}}
\nc{\ff}{\cal{F}_\Omega}
\nc{\tm}{T_\Omega}
\nc{\calp}{\cal{P}}
\makeatletter
\nc\dd{\@ifnextchar'{\ddA}{\ddB}}
\def\ddA'#1;{\rhd'_{#1\,}}
\def\ddB#1;{\rhd_{#1\,}}
\nc{\pbt}{\mathrm{PBT}}
\nc{\ad}{\mathrm{ad}}

\makeatother

\begin{document}

\title[Twisted pre-Lie algebras of finite topological spaces]{Twisted pre-Lie algebras of finite topological spaces}
\thispagestyle{empty}
\author{Mohamed Ayadi}
\address{Laboratoire de Math\'ematiques Blaise Pascal,
CNRS--Universit\'e Clermont-Auvergne,
3 place Vasar\'ely, CS 60026,
F63178 Aubi\`ere, France, and University of Sfax, Faculty of Sciences of Sfax,
LAMHA, route de Soukra,
3038 Sfax, Tunisia.}
\email{mohamed.ayadi@etu.uca.fr}

\tikzset{
			stdNode/.style={rounded corners, draw, align=right},
			greenRed/.style={stdNode, top color=green, bottom color=red},
			blueRed/.style={stdNode, top color=blue, bottom color=red}
		} 
		
	
\begin{abstract}
In this paper, we first study the species of finite topological spaces recently considered by F. Fauvet, L. Foissy, and D. Manchon.
Then, we construct a twisted pre-Lie structure on the species of connected finite topological spaces.
The underlying pre-Lie structure defines a coproduct on the species of finite topological spaces different from those already defined by the Authors above. In the end, we illustrate the link between the Grossman-Larson product and the proposed coproduct.
	\end{abstract}
\keywords{Finite topological spaces, Species, Bimonoids, Bialgebras, Hopf algebras.}
\subjclass[2010]{16T05, 16T10, 16T15, 16T30, 06A11.}
\maketitle
\tableofcontents
\section{Introduction}
A finite topological space is a finite set E endowed with a preorder $\leq$. The study of finite topological spaces was initiated by Alexandroff in 1937 \cite{acg.11}, and revived at several periods since then, using the following well-known bijection \cite{acg.60, acg...14}.
Any topology $\mathcal{T}$ on $X$ defines a quasi-order (i.e. a reflexive transitive relation) denoted by $\leq_{\mathcal{T}}$ on $X$:
	\begin{equation}
	x\leq_{\mathcal{T}}y\Longleftrightarrow \hbox{ any open subset containing $x$ also contains $y$}.
	\end{equation}
	Conversely, any quasi-order $\leq$ on $X$ defines a topology $\mathcal{T}_{\leq}$ given by its upper ideals, i.e. subsets $Y\subset X$ such that ($y\in Y$ and $y\leq z$) $\implies z\in Y$. Both operations are inverse to each other:
	\begin{equation}
	\leq_{\mathcal{T}_{\leq}}= \leq,\hspace*{2cm} \mathcal{T}_{\leq_{\mathcal{T}}}=\mathcal{T}.
	\end{equation}
	Hence there is a natural bijection between topologies and quasi-orders on
	a finite set $X$.
	Any quasi-order (hence any topology $\mathcal{T}$ ) on $X$ gives rise to an equivalence relation:
	\begin{equation}
	x \sim_{\mathcal{T}}y\Longleftrightarrow \left( x\leq_{\mathcal{T}}y \hbox{ and } y\leq_{\mathcal{T}}x \right).
	\end{equation}
	 Let $\mathcal{T}$ and $\mathcal{T}'$ be two topologies on a finite set $X$. We say that $\mathcal{T}'$ is finer than $\mathcal{T}$, and we write $\mathcal{T}'\prec \mathcal{T}$, when any open subset for $\mathcal{T}$ is an open subset for $\mathcal{T}'$. This is equivalent to the fact that for any $x,y\in X$, $x\le_{\mathcal{T}'}y\Rightarrow x\le_{\mathcal{T}}y$.\\
	
	The \textsl{quotient} $\mathcal{T}/\mathcal{T}'$ of two topologies $\mathcal{T}$ and $\mathcal{T}'$ with $\mathcal{T}'\prec \mathcal{T}$ is defined as follows (\cite[Paragraph 2.2]{acg6}): The associated quasi-order $\le_{\mathcal{T}/\mathcal{T}'}$ is the transitive closure of the relation $\mathcal{R}$ defined by:
	\begin{equation}
	x\mathcal{R} y\Longleftrightarrow (x\leq_{\mathcal{T}} y\hbox{ or }y\leq_{\mathcal{T}'} x).
	\end{equation}
	\\
More on finite topological spaces can be found in \cite{acg1, acg4, acg8, acg10, acg..12}.\\
\\
	Recall that a linear (tensor) species is a contravariant functor from the category
	of finite sets $\mathbf{Fin}$ with bijections into the category $\mathbf{Vect}$ of vector spaces (on some
	field $\mathbf{k}$). The tensor product of two species  $\mathbb{E}$ and $\mathbb{F}$ is given by
	\begin{equation}
	(\mathbb{E}\otimes \mathbb{F})_X=\bigoplus_{Y\sqcup Z= X}\mathbb{E}_{Y}\otimes\mathbb{F}_{Z},
	\end{equation}
	where the notation $\sqcup $ stands for disjoint union.
	The unit of the tensor product denoted by $\mathbf{1}$ is defined by $ \mathbf{1} _ {\varnothing} = \mathbf{k} $ and $ \mathbf{1}_X = \{0\} $, if $X\ne \varnothing$.\\
	We write $x\in \mathbb{E}$ if there exists a finite set $X$ such that $x\in \mathbb{E}_X$.\\
\\
A twisted algebra \cite{acg...3} is an algebra in the linear symmetric monoidal category of linear species. See \cite{acg...2, acg.40, acg...1} for further details on and references to Joyal’s theory of twisted algebras. Concretely, a twisted algebra is a linear species $\mathbb{E}$ provided with a product map (which
is a map of linear species: $\mathbb{E} \otimes \mathbb{E} \to \mathbb{E}$). Associative algebras, commutative algebras,
Lie algebras, pre-Lie algebras and so on, are defined accordingly.
\\
\\
The species $\mathbb{T}$ of finite topological spaces is defined as follows: For any finite set $X$, $\mathbb{T}_X$  is
	the vector space freely generated by the topologies on $X$. For any bijection
	$\varphi : X \longrightarrow  X^{\prime } $, the isomorphism $\mathbb{T}_{\varphi } : \mathbb{T}_{ X^{\prime}} \longrightarrow \mathbb{T}_X$ is defined by the obvious
	relabelling:
	$$\mathbb{T}_{\varphi }(\mathcal{T})=\{ \varphi^{-1}(Y), Y\in  \mathcal{T}  \},$$
	for any topology $\mathcal{T}$ on $X^{\prime }$.\\
	\\
		A unital associative algebra (\cite[Paragraph 2.3]{acg6}) on
	the species of finite topologies is defined as follows: for any pair $X_1, X_2$ of
	finite sets we introduce
	\begin{align*}
	m:\mathbb{T}_{X_1}\otimes \mathbb{T}_{X_2} \longrightarrow \mathbb{T}_{X_1\sqcup X_2} \\
	\mathcal{T}_1\otimes \mathcal{T}_2\longmapsto \mathcal{T}_1\mathcal{T}_2,
	\end{align*}
	where $\mathcal{T}_1\mathcal{T}_2$ is the disjoint union topology characterised by $Y\in \mathcal{T}_1\mathcal{T}_2$ if and only if $Y\cap X_1\in \mathcal{T}_1$ and
	$Y\cap X_2\in \mathcal{T}_2$. The unit is given by the unique topology on the empty set.\\
	\\
		For any topology $\mathcal{T}$ on a finite set $X$ and for any subset $Y \subset X$, we
	denote by $\mathcal{T}_{|Y}$ the restriction of $\mathcal{T}$ to $Y$. It is defined by:
	$$\mathcal{T}_{|Y}= \left\lbrace Z\cap Y, Z\in \mathcal{T} \right\rbrace. $$
	The external coproduct $\Delta$ on $\mathbb{T}$ is defined as follows:
	\begin{eqnarray*}
		\Delta:\mathbb{T}_X&\longrightarrow& (\mathbb{T}\otimes \mathbb{T})_X=\bigoplus_{Y\sqcup Z= X}\mathbb{T}_{Y}\otimes\mathbb{T}_{Z}\\
		\mathcal{T}&\longmapsto& \sum_{Y\in  \mathcal{T}}\mathcal{T}_{|X\backslash Y}\otimes \mathcal{T}_{|Y}.
	\end{eqnarray*}
	The species $\mathbb{T}$ is this way endowed with a twisted bialgebra structure in \cite{acg6}. \\

		\noindent Now consider the graded vector space:
	\begin{equation}
	\mathcal{H}=\overline{\mathcal{K}}(\mathbb{T})=\bigoplus \limits_{\underset{}{n\geq 0}} \mathcal{H}_n
	\end{equation}
	where $\mathcal{H}_0 =\mathbf{k}.1$, and where $\mathcal{H}_n$ is the linear span of topologies on $\left\{ 1, . . . , n \right\}$
	when $n \geq 1$, modulo the action of the symmetric group $S_n$. The vector space $\mathcal{H}$ can be seen as the quotient of the species $\mathbb{T}$ by the "forget the labels" equivalence relation: $\mathcal{T} \sim \mathcal{T}^{\prime}$  if $\mathcal{T}$ $\left(\text{resp.} \mathcal{T}^{\prime} \right)$ is a topology on a finite
	set $X$ (resp. $X'$), such that there is a bijection from $X$ onto $X^{\prime}$ which is a homeomorphism with respect to both topologies. The functor $\overline{\mathcal{K}}$ from linear species to graded vector spaces thus obtained is intensively studied in (\cite[chapter 15]{acg10}) under the name "bosonic Fock functor". The twisted Hopf algebra structure on $\mathbb{T}$ \cite{acg6}
	naturally leads to the following:\\
	\\
	 $(\mathcal{H},m,\Delta )$ is a commutative connected Hopf algebra, graded by the number of elements.\\

   L. Foissy, C. Malvenuto and F. Patras in \cite[section 6]{acg...00} were the first to prove that the finite topological spaces can be organized in a graded commutative Hopf algebra. The latter can be recovered by applying the $\overline{\mathcal{K}}$ functor to the twisted Hopf algebra structure on $\mathbb{T}$ described in \cite{acg6}. The coproduct $\Delta$ defined therein is however not built from a pre-Lie structure. We define in the present work two twisted pre-Lie structures $\searrow$ and $\nearrow$ on the species of connected finite topological spaces, giving rise to two more coproducts $\Delta_{\searrow}$ and $\Delta_{\nearrow}$, hence two more twisted Hopf algebra structures. We expect that this will contribute  to a better understanding of the finite topological spaces considered as a whole.\\

   In section \ref{enveloping}, we recall the method of D. Guin and J.-M. Oudom  \cite{acg9} to describe the enveloping
algebra of a pre-Lie algebra, and we adapt it to the twisted context, following indications in \cite{acg.20}.\\

   In Section \ref{The enveloping algebra of topologie} of this paper, we define the enveloping algebra of the grafting twisted pre-Lie algebra of connected finite topological spaces, as well as its enveloping algebra using the Guin-Oudom method. Denoting by $\mathbb{V}$ the species of connected finite topological spaces, we consider the Hopf symmetric algebra
$\mathcal{H}'=S(\mathbb{V})$ of the pre-Lie twisted algebra $(\mathbb{V}, \searrow)$, equipped with its usual unshuffling coproduct $\Delta_{unsh}$
and a product $\star$ defined on $\mathbb{T}$ by: For any pair $X_1, X_2$ of finite sets
\begin{eqnarray*}
		\star:\mathbb{T}_{X_1}\otimes\mathbb{T}_{X_2}&\longrightarrow& \mathbb{T}_{X_1\sqcup X_2}\\
		(\mathcal{T}_1, \mathcal{T}_2)&\longmapsto& \sum_{(\mathcal{T}_1)}\mathcal{T}^{(1)}_1(\mathcal{T}^{(2)}_1\searrow \mathcal{T}_2).
	\end{eqnarray*}
	
      In section \ref{bialgebras} we prove that there exists a twisted bialgebra structure on $\mathbb{T}$, where the external coproduct is defined by 
      \begin{eqnarray*}
		\Delta_{\searrow}:\mathbb{T}_X&\longrightarrow& (\mathbb{T}\otimes \mathbb{T})_X=\bigoplus_{Y\sqcup Z= X}\mathbb{T}_{Y}\otimes\mathbb{T}_{Z}\\
		\mathcal{T}&\longmapsto& \sum_{Y \overline{\in}  \mathcal{T}} \mathcal{T}_{|Y} \otimes \mathcal{T}_{|X\backslash Y}
	\end{eqnarray*}
	where $Y \overline{\in}  \mathcal{T}$ stands for
\begin{itemize}
    \item $Y\in \mathcal{T}$,
    \item $\mathcal{T}_{|Y}=\mathcal{T}_1...\mathcal{T}_n$, such that for all $i\in \{1,..., n\}, \mathcal{T}_i$ connected and \big( $\hbox{min}\mathcal{T}_i=(\hbox{min}\mathcal{T})\cap \mathcal{T}_i$, or there is a single common ancestor $x_i \in \overline{X\backslash Y}$ to $\hbox{min}\mathcal{T}_i$ \big), where $\overline{X\backslash Y}=(X\backslash Y)/\sim_{\mathcal{T}_{|X\backslash Y}}$.
\end{itemize}
We moreover give a relation between the two structures $\Delta_{\searrow}$ and $\star$.\\

     Finally we define in section \ref{Relation between} a new pre-Lie law $\nearrow$ on the species of connected finite topological spaces by: 
  For all $\mathcal{T}=(X, \leq_{\mathcal{T}})$ and $\mathcal{S}=(Y, \leq_{\mathcal{S}})$ be two finite topological spaces,
 \begin{align*}
  \mathcal{T}\nearrow \mathcal{S}:=j\big( j(\mathcal{T})\searrow j(\mathcal{S}) \big),  
\end{align*}
 where $j$ is the involution which transforms $\le$ into $\ge$. This law $\nearrow$ gives rise to a coproduct denoted $\Delta_{\nearrow}$ defined by $ \Delta_{\nearrow}= (j \otimes j)\Delta_{\searrow}\circ j $.\\
  For any finite set $A$ and for any pair of parts $A_1$, $A_2$ of $A$ with $A_1\cap A_2=\varnothing$, we define
     $$\Psi_{A_1,\, A_2}:\mathbb T_A\to\mathbb T_A,$$
     as follows: for any topology $\mathcal{T}\in \mathbb{T}_A$, the topology $\Psi_{A_1,\,  A_2}(\mathcal T)$ is associated with the following pre-order $\le$ defined by:
     \begin{itemize}
    \item If $a\in A_1$, and $b\in A_2$ then $a$ and $b$ are incomparable,
    \item If not, we have: $a\le b$ if and only if $a\le_{\mathcal T} b$.
    \end{itemize} 
  
  In this section, we provide a relation between both pre-Lie structures, by proving that the following diagram commutes.
$$
\xymatrix{
	\mathbb{V}_X \otimes \mathbb{V}_Y \otimes \mathbb{V}_Z \ar[rr]^{id\otimes \searrow_u} \ar[d]_{\nearrow^s\otimes id} && \mathbb{V}_X \otimes \mathbb{V}_{Y\sqcup Z} \ar[dd]^{\nearrow^s}\\
			\mathbb{V}_{X\sqcup Y} \otimes \mathbb{V}_Z \ar[d]_{\searrow_u } && \\
			\mathbb{V}_{X\sqcup Y\sqcup Z} \ar[rr]_{\Psi_{X,\, Z}} && \mathbb{V}_{X\sqcup Y\sqcup Z}
}
$$
\section{The enveloping algebra of pre-Lie algebras and twisted pre-Lie algebras}\label{enveloping}
	In this section, we recall the method of D. Guin and J.-M. Oudom  \cite{acg9} to describe the enveloping
algebra of a pre-Lie algebra. We also recall how T. Schedler in \cite{acg.20} generalizes this method to twisted pre-Lie algebras.
\begin{definition}
    A Lie algebra over a field $\mathbf{k}$ is a vector space $V$ endowed with a bilinear bracket $[., .]$ satisfying:\\
    (1) The antisymmetry:
    $$[x,y]=-[y,x], \forall x, y \in V.$$
    (2) The Jacobi identity:
    $$[x,[y,z]]+[y,[z,x]]+[z,[x,y]]=0, \forall x, y, z \in V.$$
\end{definition}
\begin{definition}
    \cite{acg4, acg3} A left pre-Lie algebra over a field $\mathbf{k}$ is a $\mathbf{k}$-vector space $A$ with a binary composition $\rhd$ that satisfies the left pre-Lie identity:
    \begin{equation*}
        (x\rhd y)\rhd z-x\rhd (y\rhd z)=(y\rhd x)\rhd z-y\rhd (x\rhd z),
    \end{equation*}
    for all $x, y, z\in A$. The left pre-Lie identity rewrites
as:
\begin{equation}
L_{[x,y]}=[L_x, L_y],
\end{equation}
where $L_x:A\longrightarrow A$ is defined by $L_xy=x\rhd y$, and where the bracket on the left-hand side is
defined by $[x,y]=x\rhd y-y\rhd x$. As a consequence this bracket satisfies the Jacobi identity.
\end{definition}
  The pre-Lie product is extended to the symmetric algebra as follows \cite{acg7}. Let $(A,\rhd)$ be a pre-Lie algebra. We consider the Hopf symmetric algebra $S(A)$ equipped with its usual unshuffle coproduct denoted $\Delta_{unsh}$. We will use without restraint the classical Sweedler notation: $\Delta_{unsh}(a)=\sum \limits_{\underset{}{a}}a^{(1)}\otimes a^{(2)}.$\\
    We extend the product $\rhd$ to $S(A)$. Let $a, b$ and
    $c \in S(A)$, and $x\in A$. We put:
  \begin{itemize}
        \item $1\rhd a=a$
        \item $a\rhd 1=\varepsilon (a)1$
        \item $(xa)\rhd b=x\rhd (a\rhd b)-(x\rhd a)\rhd b$
        \item $a\rhd (bc)=\sum \limits_{\underset{}{a}}(a^{(1)}\rhd b)(a^{(2)}\rhd c)$.
   \end{itemize}
   \begin{definition}\label{def.}
    We define the following $\star$ product on $S(A)$ by:
\begin{equation}
    a\star b=\sum_{a}a^{(1)}(a^{(2)}\rhd b).
\end{equation}
\end{definition}
\begin{theorem}
    \cite{acg..15, acg9} The triple $(S(A), \star, \Delta_{unsh})$ is a Hopf algebra which is isomorphic to the enveloping Hopf algebra $\mathcal{U}(A_{Lie})$ of the Lie algebra $A_{Lie}$.
\end{theorem}
\begin{proof}
    This theorem was proved by D. Guin and J.-M. Oudom  in \cite{acg9} (Lemma 2.10 and Theorem 2.12).
\end{proof}
\begin{definition}
   \cite{acg.40} A twisted Lie algebra over a field $\mathbf{k}$, is a species $\mathbb{E}$ endowed with a bilinear bracket $[,]:\mathbb{E}\otimes \mathbb{E}\to \mathbb{E}$, satisfying:\\
    (i) $[ , ]+[ , ]\tau =0,$\\
    (ii) $[ , [ , ] ]+[ ,[ , ] ]\Sigma +[ ,[ , ] ]\Sigma^2=0$,\\
    where $\tau:\mathbb{E}\otimes \mathbb{E}\to \mathbb{E}\otimes \mathbb{E}$ is the flip, and $\Sigma:\mathbb{E}\otimes \mathbb{E}\otimes \mathbb{E}\to \mathbb{E}\otimes \mathbb{E}\otimes \mathbb{E}$ is the cyclic permutation of factors.
\end{definition}
\begin{definition}
    A left twisted pre-Lie algebra over a field $\mathbf{k}$, is a  species $\mathbb{E}$ with a binary composition
    $\circ:\mathbb{E}\otimes \mathbb{E}\to \mathbb{E}$, satisfing the left twisted pre-Lie algebra identity
   $$ \circ (\circ \otimes Id)-\circ (Id\otimes \circ)=\big( \circ (\circ \otimes Id)-\circ (Id\otimes \circ) \big)(\tau \otimes Id).$$
    \end{definition}
    
      T. Schedler in \cite{acg.20} shows that the properties of D. Guin and J.-M. Oudom above also work for the linear  species, i.e:\\
Let $(\mathbb{E},\circ)$ be a twisted pre-Lie algebra. We consider the twisted Hopf symmetric algebra $S(\mathbb{E})$ equipped with its usual unshuffle coproduct denoted $\Delta_{unsh}$. We extend the product $\circ$ to $S(\mathbb{E})$ as follows. Let $a, b$ and
    $c \in S(\mathbb{E})$, and $x\in \mathbb{E}$. We put:
  \begin{itemize}
        \item $1\circ a=a$
        \item $(xa)\circ b=x\circ (a\circ b)-(x\circ a)\circ b$
        \item $a\circ (bc)=\sum \limits_{\underset{}{a}}(a^{(1)}\circ b)(a^{(2)}\circ c)$,
   \end{itemize}
and if we define the product $\star$ on $S(\mathbb{E})$ by:
\begin{equation}
    a\star b=\sum_{a}a^{(1)}(a^{(2)}\circ b),
\end{equation}
then $(S(\mathbb{E}), \star, \Delta_{unsh})$ is isomorphic to the enveloping Hopf algebra $\mathcal{U}(\mathbb{E}_{Lie})$ of the twisted Lie algebra $\mathbb{E}_{Lie}$.
\section{The enveloping algebra of the twisted pre-Lie algebra of finite topological spaces}\label{The enveloping algebra of topologie}
\subsection{The pre-Lie algebra of rooted trees}
	Let $T$ the vector space spanned by the set of isomorphism classes of rooted trees and $\mathcal{H}=S(T)$.
	Grafting pre-Lie algebras of rooted trees were studied for the first time by F. Chapoton and M. Livernet \cite{acg5}, see also D. Manchon and A. Saidi \cite{acg1}. The grafting product is given, for all $t, s \in T$, by:
	\begin{equation}
	    t\rightarrow s=\sum_{s' \hbox{ \tiny{vertex of} }  s} t\rightarrow_{s'}s,
	\end{equation}
	where $t\rightarrow_{s'} s$ is the tree obtained by grafting the root of $t$ on the vertex $s'$ of $s$. More explicitly, the operation $t\rightarrow s$ consists of grafting the root of $t$ on every vertex of $s$ and summing up.
	\begin{theorem}\cite{acg5}
	Equipped by $\rightarrow$, the space $T$ is the free pre-Lie algebra with one generator.
	\end{theorem}
		Now, we can use the method of D. Guin and J.-M. Oudom  \cite{acg9} to find the enveloping algebra of the grafting pre-Lie algebra of rooted trees. We consider the Hopf symmetric algebra $\mathcal{H}=S(T)$
of the pre-Lie algebra $(T, \rightarrow)$, equipped with its usual unshuffling coproduct $\Delta_{unsh}$. We extend the
product $\rightarrow$ to $\mathcal{H}$ by the same method used in (3.1), and we define the Grossman-Larson product \cite{acg50} $\star$ on $\mathcal{H}$ by:
$$t\star t'=\sum_{t}t^{(1)}(t^{(2)}\rightarrow t').$$
By construction, the space $(\mathcal{H}, \star, \Delta_{unsh})$ is a Hopf algebra.
\subsection{ Twisted pre-Lie algebra of the finite topological spaces}
Let $\mathcal{T}_1=(X_1, \leq_{\mathcal{T}_1})$ and $\mathcal{T}_2=(X_2, \leq_{\mathcal{T}_2})$ be two finite topological spaces, and let $v\in X_2$. We define:
\begin{align*}
  \mathcal{T}_1\searrow_v \mathcal{T}_2:=(X_1\sqcup X_2, \leq),  
\end{align*}
where $\leq$ is obtained from $\leq_{\mathcal{T}_1}$ and $\leq_{\mathcal{T}_2}$ as follows: compare any pair in $X_2$ (resp. $X_1$)
by using $\leq_{\mathcal{T}_2}$ (resp. $\leq_{\mathcal{T}_1}$), and compare any element $y\in X_2$ with any element $x\in X_1$.\\
\\
To sum up, for any $x, y \in X_1\sqcup X_2$, $x\leq y$ if and only if:
\begin{itemize}
    \item Either $x, y \in X_1$ and $x\leq_{\mathcal{T}_1}y$,
    \item or $x,y \in X_2$ and $x\leq_{\mathcal{T}_2}y$,
    \item or $x\in X_2$, $y\in X_1$ and $x\leq_{\mathcal{T}_2} v$.
\end{itemize}
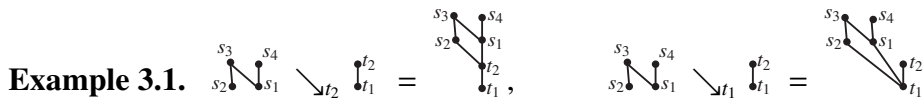
\begin{example}

$\fcolorbox{white}{white}{
\scalebox{0.7}{
  \begin{picture}(26,39) (415,-324)
    \SetWidth{1.0}
    \SetColor{Black}
    \Vertex(433,-321){2}
    \Vertex(433,-309){2}
    \Vertex(418,-308){2}
    \Vertex(419,-321){2}
    \Line(419,-321)(418,-309)
    \Line(433,-322)(417,-308)
    \Text(409,-325)[lb]{\small{\Black{$s_2$}}}
    \Text(437,-325)[lb]{\small{\Black{$s_1$}}}
    \Text(412,-306)[lb]{\small{\Black{${s_3}$}}}
    \Text(436,-308)[lb]{\small{\Black{$s_4$}}}
    \Line(433,-322)(433,-309)
  \end{picture}
}} \searrow_{t_2} \fcolorbox{white}{white}{
\scalebox{0.7}{
  \begin{picture}(0,27) (443,-336)
    \SetWidth{1.0}
    \SetColor{Black}
    \Vertex(433,-333){2}
    \Vertex(433,-321){2}
    \Text(437,-325)[lb]{\small{\Black{$t_2$}}}
    \Text(437,-338)[lb]{\small{\Black{$t_1$}}}
    \Line(433,-334)(433,-321)
  \end{picture}
}}=\fcolorbox{white}{white}{
\scalebox{0.7}{
  \begin{picture}(23,64) (417,-324)
    \SetWidth{1.0}
    \SetColor{Black}
    \Vertex(433,-322){2}
    \Vertex(433,-310){2}
    \Vertex(433,-296){2}
    \Vertex(433,-284){2}
    \Vertex(418,-283){2}
    \Vertex(419,-296){2}
    \Line(433,-323)(433,-284)
    \Line(434,-311)(419,-297)
    \Line(419,-296)(418,-284)
    \Line(433,-297)(417,-283)
    \Text(437,-328)[lb]{\small{\Black{$t_1$}}}
    \Text(437,-316)[lb]{\small{\Black{$t_2$}}}
    \Text(409,-302)[lb]{\small{\Black{$s_2$}}}
    \Text(437,-302)[lb]{\small{\Black{$s_1$}}}
    \Text(408,-287)[lb]{\small{\Black{$s_3$}}}
    \Text(437,-288)[lb]{\small{\Black{$s_4$}}}
  \end{picture}
}},\hspace*{1cm}\fcolorbox{white}{white}{
\scalebox{0.7}{
  \begin{picture}(26,39) (415,-324)
    \SetWidth{1.0}
    \SetColor{Black}
    \Vertex(433,-321){2}
    \Vertex(433,-309){2}
    \Vertex(418,-308){2}
    \Vertex(419,-321){2}
    \Line(419,-321)(418,-309)
    \Line(433,-322)(417,-308)
    \Text(409,-325)[lb]{\small{\Black{$s_2$}}}
    \Text(437,-325)[lb]{\small{\Black{$s_1$}}}
    \Text(412,-306)[lb]{\small{\Black{$s_3$}}}
    \Text(436,-308)[lb]{\small{\Black{$s_4$}}}
    \Line(433,-322)(433,-309)
  \end{picture}
}} \searrow_{t_1}  \fcolorbox{white}{white}{
\scalebox{0.7}{
  \begin{picture}(-3,27) (444,-336)
    \SetWidth{1.0}
    \SetColor{Black}
    \Vertex(433,-333){2}
    \Vertex(433,-321){2}
    \Text(437,-325)[lb]{\small{\Black{$t_2$}}}
    \Text(437,-338)[lb]{\small{\Black{$t_1$}}}
    \Line(433,-334)(433,-321)
  \end{picture}
}}=\fcolorbox{white}{white}{
\scalebox{0.7}{
  \begin{picture}(73,63) (417,-324)
    \SetWidth{0.0}
    \SetColor{Black}
    \Vertex(433,-297){2}
    \Vertex(433,-285){2}
    \Vertex(418,-284){2}
    \Vertex(419,-297){2}
    \SetWidth{1.0}
    \Line(419,-297)(418,-285)
    \Line(433,-298)(417,-284)
     \Text(409,-302)[lb]{\small{\Black{$s_2$}}}
    \Text(437,-302)[lb]{\small{\Black{$s_1$}}}
    \Text(408,-287)[lb]{\small{\Black{$s_3$}}}
    \Text(437,-288)[lb]{\small{\Black{$s_4$}}}
    \Line(433,-297)(432,-284)
    \Vertex(449,-321){2}
    \Line(448,-321)(419,-298)
    \Line(449,-321)(433,-298)
    \Vertex(449,-309){2}
    \Line(449,-321)(449,-310)
    \Text(454,-328)[lb]{\small{\Black{$t_1$}}}
    \Text(454,-314)[lb]{\small{\Black{$t_2$}}}
  \end{picture}
}}$

\end{example}
\begin{proposition}
Let $\mathcal{T}_1=(X_1, \leq_{\mathcal{T}_1})$ and $\mathcal{T}_2=(X_2, \leq_{\mathcal{T}_2})$ be two connected finite topological spaces, and let $v\in X_2$. Then $\mathcal{T}_1\searrow_v \mathcal{T}_2:=(X_1\sqcup X_2, \leq)$, is a connected finite topological space.
\end{proposition}
\begin{proof}
Let $\mathcal{T}_1=(X_1, \leq_{\mathcal{T}_1})$ and $\mathcal{T}_2=(X_2, \leq_{\mathcal{T}_2})$ be two connected finite topological spaces, and let $v\in X_2$.\\
We must show that $\leq$ is a preorder relation on $X_1\sqcup X_2$:\\
\textbf{Reflexivity}; Let $x\in X_1\sqcup X_2$, then $x\in X_1$ or $x\in X_2$.\\
If $x\in X_1$, we have $x\leq_{\mathcal{T}_1} x$, then $x\leq x$.\\
If $x\in X_2$, we have $x\leq_{\mathcal{T}_2} x$, then $x\leq x$.\\
\textbf{Transitivity}; Let $x, y, z\in X_1\sqcup X_2$ such that $x\leq y$ and $y\leq z$. So we have four possible cases:
\begin{itemize}
    \item First case; $x, y, z\in X_1$, and ($x\leq_{\mathcal{T}_1} y$ and $y\leq_{\mathcal{T}_1} z$).\\
Since $\leq_{\mathcal{T}_1}$ is transitive, then $x\leq_{\mathcal{T}_1} z$, then $x\leq z$.
    \item Second case; $x, y, z\in X_2$, and ($x\leq_{\mathcal{T}_2} y$ and $y\leq_{\mathcal{T}_2} z$).\\
Since $\leq_{\mathcal{T}_2}$ is transitive, then $x\leq_{\mathcal{T}_2} z$, then $x\leq z$.
    \item Third case; $x, y\in X_2$, $z\in X_1$, and ($x\leq_{\mathcal{T}_2} y$ and $y\leq_{\mathcal{T}_2} v$).\\
Since $\leq_{\mathcal{T}_2}$ is transitive, then $x\leq_{\mathcal{T}_2} v$, and since $x\in X_2$ and $z\in X_1$ therefore $x\leq z$.
    \item Fourth case; $x\in X_2$, $y, z\in X_1$, and ($x\leq_{\mathcal{T}_2} v$ and $y\leq_{\mathcal{T}_1} z$).\\
In this case we have $x\in X_2$, $z\in X_1$, and $x\leq_{\mathcal{T}_2} v$, then $x\leq z$.
\end{itemize}
\end{proof}
\begin{proposition}
Let $\mathcal{T}_1=(X_1, \leq_{\mathcal{T}_1})$, $\mathcal{T}_2=(X_2, \leq_{\mathcal{T}_2})$ and $\mathcal{T}_3=(X_3, \leq_{\mathcal{T}_3})$ be three finite connected topological spaces, and let $u\in X_2$, $v, w\in X_3$. Then\\
1) $(\mathcal{T}_1\searrow_u \mathcal{T}_2)\searrow_w \mathcal{T}_3=\mathcal{T}_1\searrow_u(\mathcal{T}_2\searrow_w \mathcal{T}_3)$.\\
2) $\mathcal{T}_1\searrow_v(\mathcal{T}_2\searrow_w \mathcal{T}_3)=\mathcal{T}_2\searrow_w(\mathcal{T}_1\searrow_v \mathcal{T}_3)$.
\end{proposition}
\begin{proof}
1) Let $\mathcal{T}_1=(X_1, \leq_{\mathcal{T}_1})$, $\mathcal{T}_2=(X_2, \leq_{\mathcal{T}_2})$ and $\mathcal{T}_3=(X_3, \leq_{\mathcal{T}_3})$ be three finite connected topologies, and let $u\in X_2$, 
$w\in X_3$. 
We denote $\mathcal{T}'_3=(\mathcal{T}_1\searrow_u \mathcal{T}_2)=(X_1\sqcup X_2, \leq_3)$,
with $\leq_3$ defined on $X_1\sqcup X_2$ by:\\
$x, y\in X_1\sqcup X_2$ et $x\leq_3 y$ if and only if:
\begin{itemize}
    \item Either $x, y \in X_1$ and $x\leq_{\mathcal{T}_1}y$,
    \item or $x, y \in X_2$ and $x\leq_{\mathcal{T}_2}y$,
    \item or $x\in X_2$, $y\in X_1$ and $x\leq_{\mathcal{T}_2} u$,
\end{itemize}
and we denote
$\mathcal{T}=(\mathcal{T}_1\searrow_u \mathcal{T}_2)\searrow_w \mathcal{T}_3=(X_1\sqcup X_2\sqcup X_3, \leq)$, with $\leq$ defined on $X_1\sqcup X_2\sqcup X_3$ by:\\
$x, y\in X_1\sqcup X_2\sqcup X_3$ et $x\leq y$ if and only if:
\begin{itemize}
    \item Either $x, y \in X_1\sqcup X_2$ and $x\leq_3y$,
    \item or $x, y \in X_3$ and $x\leq_{\mathcal{T}_3}y$,
    \item or $x\in X_3$, $y\in X_1\sqcup X_2$ and $x\leq_{\mathcal{T}_3} w$,
\end{itemize}
then\\
$x, y\in X_1\sqcup X_2\sqcup X_3$ et $x\leq y$ if and only if:
\begin{itemize}
    \item Either $x, y \in X_1$ and $x\leq_{\mathcal{T}_1}y$,
    \item or $x,y \in X_2$ and $x\leq_{\mathcal{T}_2}y$,
    \item or $x\in X_2$, $y\in X_1$ and $x\leq_{\mathcal{T}_2} u$,
    \item or $x,y \in X_3$ and $x\leq_{\mathcal{T}_3}y$,
    \item or $x\in X_3$, $y\in X_1\sqcup X_2$ and $x\leq_{\mathcal{T}_3} w$.
\end{itemize}
On the other hand,\\
we denote $\mathcal{T}'_1=\mathcal{T}_2\searrow_w \mathcal{T}_3=(X_2\sqcup X_3, \leq_1)$,
with $\leq_1$ defined on $X_2\sqcup X_3$ by:\\
$x, y\in X_2\sqcup X_3$ et $x\leq_1 y$ if and only if:
\begin{itemize}
    \item Either $x, y \in X_2$ and $x\leq_{\mathcal{T}_2}y$,
    \item or $x, y \in X_3$ and $x\leq_{\mathcal{T}_3}y$,
    \item or $x\in X_3$, $y\in X_2$ and $x\leq_{\mathcal{T}_3} w$,
\end{itemize}
and we denote
$\mathcal{T}'=\mathcal{T}_1\searrow_u(\mathcal{T}_2\searrow_w \mathcal{T}_3)=(X_1\sqcup X_2\sqcup X_3, \leq')$, with $\leq'$ defined on $X_1\sqcup X_2\sqcup X_3$ by:\\
$x, y\in X_1\sqcup X_2\sqcup X_3$ et $x\leq' y$ if and only if:
\begin{itemize}
    \item Either $x, y \in X_1$ and $x\leq_{\mathcal{T}_1}y$,
    \item or $x, y \in X_2\sqcup X_3$ and $x\leq_1 y$,
    \item or $x\in X_2\sqcup X_3$, $y\in X_1$ and $x\leq_1 u$,
\end{itemize}
then \\
$x, y\in X_1\sqcup X_2\sqcup X_3$ et $x\leq' y$ if and only if:
\begin{itemize}
    \item Either $x, y \in X_1$ and $x\leq_{\mathcal{T}_1}y$,
    \item or $x, y \in X_2$ and $x\leq_{\mathcal{T}_2}y$,
    \item or $x, y \in X_3$ and $x\leq_{\mathcal{T}_3}y$,
    \item or $x\in X_3$, $y\in X_2$ and $x\leq_{\mathcal{T}_3} w$,
    \item or $x\in X_2$, $y\in X_1$ and $x\leq_{\mathcal{T}_2} u$,
    \item or $x\in X_3$, $y\in X_1$ and $x\leq_{\mathcal{T}_3} w$,
\end{itemize}
then $\leq=\leq'$ on $X_1\sqcup X_2\sqcup X_3$.\\
Then
$$(\mathcal{T}_1\searrow_u \mathcal{T}_2)\searrow_w \mathcal{T}_3=\mathcal{T}_1\searrow_u(\mathcal{T}_2\searrow_w \mathcal{T}_3).$$
2) Let $v, w \in X_3$, we denote $\mathcal{T}'_2=(\mathcal{T}_1\searrow_v \mathcal{T}_3)=(X_1\sqcup X_3, \leq_2)$, with $\leq_2$ defined on $X_1\sqcup X_3$ by:\\
$x, y\in X_1\sqcup X_3$ et $x\leq_2 y$ if and only if:
\begin{itemize}
    \item Either $x, y \in X_1$ and $x\leq_{\mathcal{T}_1}y$,
    \item or $x, y \in X_3$ and $x\leq_{\mathcal{T}_3}y$,
    \item or $x\in X_3$, $y\in X_1$ and $x\leq_{\mathcal{T}_3} v$,
\end{itemize}
and we denote\\
$\mathcal{T}=\mathcal{T}_2\searrow_w (\mathcal{T}_1\searrow_v \mathcal{T}_3)=(X_1\sqcup X_2\sqcup X_3, \leq)$, with $\leq$ defined on $X_1\sqcup X_2\sqcup X_3$ by:\\
$x, y\in X_1\sqcup X_2\sqcup X_3$ et $x\leq y$ if and only if:
\begin{itemize}
    \item Either $x, y\in X_2$ and $x\leq_{\mathcal{T}_2}y$,
    \item or $x, y\in X_1\sqcup X_3$ and $x\leq_2 y$,
    \item or $x\in X_1\sqcup X_3$, $y\in X_2$ and $x\leq_2 w$,
\end{itemize}
then\\
$x, y\in X_1\sqcup X_2\sqcup X_3$ et $x\leq y$ if and only if:
\begin{itemize}
    \item Either $x, y\in X_2$ and $x\leq_{\mathcal{T}_2}y$,
    \item or $x, y \in X_1$ and $x\leq_{\mathcal{T}_1}y$,
    \item or $x, y \in X_3$ and $x\leq_{\mathcal{T}_3}y$,
    \item or $x\in X_3$, $y\in X_1$ and $x\leq_{\mathcal{T}_3} v$,
    \item or $x\in X_3$, $y\in X_2$ and $x\leq_{\mathcal{T}_3} w$.
\end{itemize}
On the other hand,\\
 we denote $\mathcal{T}'_1=(\mathcal{T}_2\searrow_w \mathcal{T}_3)=(X_2\sqcup X_3, \leq_1)$
 , with $\leq_1$ defined on $X_2\sqcup X_3$ by:\\
$x, y\in X_2\sqcup X_3$ et $x\leq_1 y$ if and only if:
\begin{itemize}
    \item Either $x, y \in X_2$ and $x\leq_{\mathcal{T}_2}y$,
    \item or $x, y \in X_3$ and $x\leq_{\mathcal{T}_3}y$,
    \item or $x\in X_3$, $y\in X_2$ and $x\leq_{\mathcal{T}_3} w$,
\end{itemize}
and we denote\\
$\mathcal{T}'=\mathcal{T}_1\searrow_v (\mathcal{T}_2\searrow_w \mathcal{T}_3)=(X_1\sqcup X_2\sqcup X_3, \leq')$, with $\leq'$ defined on $X_1\sqcup X_2\sqcup X_3$ by:\\
$x, y\in X_1\sqcup X_2\sqcup X_3$ et $x\leq' y$ if and only if:
\begin{itemize}
    \item Either $x, y\in X_1$ and $x\leq_{\mathcal{T}_1}y$,
    \item or $x, y\in X_2\sqcup X_3$ and $x\leq_1 y$,
    \item or $x\in X_2\sqcup X_3$, $y\in X_1$ and $x\leq_1 v$,
\end{itemize}
then\\
$x, y\in X_1\sqcup X_2\sqcup X_3$ et $x\leq' y$ if and only if:
\begin{itemize}
    \item Either $x, y\in X_1$ and $x\leq_{\mathcal{T}_1}y$,
    \item or $x, y \in X_2$ and $x\leq_{\mathcal{T}_2}y$,
    \item or $x, y \in X_3$ and $x\leq_{\mathcal{T}_3}y$,
    \item or $x\in X_3$, $y\in X_2$ and $x\leq_{\mathcal{T}_3} w$,
    \item or $x\in X_3$, $y\in X_1$ and $x\leq_{\mathcal{T}_3} v$,
\end{itemize}
then $\leq=\leq'$ on $X_1\sqcup X_2\sqcup X_3$.\\
Then 
$$\mathcal{T}_1\searrow_v(\mathcal{T}_2\searrow_w \mathcal{T}_3)=\mathcal{T}_2\searrow_w(\mathcal{T}_1\searrow_v \mathcal{T}_3).$$
\end{proof}
We then define the grafting law in the species of connected finite topological spaces by:\\
For all $ \mathcal{T}_1 \in \mathbb{T}_{X_1,} \mathcal{T}_2 \in \mathbb{T}_{X_2}$
$$ \mathcal{T}_1\searrow \mathcal{T}_2=\sum_{v\in X_2}\mathcal{T}_1\searrow_v \mathcal{T}_2 \in \mathbb{T}_{X_1 \sqcup X_2}.$$
\begin{theorem}
$(\mathbb{V}, \searrow)$ is a twisted pre-Lie algebra. 
\end{theorem}
\begin{proof}
Let $\mathcal{T}_1=(X_1, \leq_{\mathcal{T}_1})$, $\mathcal{T}_2=(X_2, \leq_{\mathcal{T}_2})$ and $\mathcal{T}_3=(X_3, \leq_{\mathcal{T}_3})$ three finite topological spaces, we have:
\begin{align*}
   \mathcal{T}_1\searrow(\mathcal{T}_2\searrow \mathcal{T}_3)&=\sum_{v\in X_3}\mathcal{T}_1 \searrow(\mathcal{T}_2\searrow_v \mathcal{T}_3)\\
   &=\sum_{u\in X_2 \sqcup X_3}\sum_{v\in X_3}\mathcal{T}_1 \searrow_u(\mathcal{T}_2\searrow_v \mathcal{T}_3)\\
   &=\sum_{u\in X_2}\sum_{v\in X_3}\mathcal{T}_1 \searrow_u(\mathcal{T}_2\searrow_v \mathcal{T}_3)\\
   &\hspace{0.5cm}+\sum_{u\in X_3}\sum_{v\in X_3}\mathcal{T}_1 \searrow_u(\mathcal{T}_2\searrow_v \mathcal{T}_3).
\end{align*}
On the other hand we have:
\begin{align*}
   (\mathcal{T}_1\searrow \mathcal{T}_2)\searrow \mathcal{T}_3&= \sum_{r\in X_2}(\mathcal{T}_1\searrow_r \mathcal{T}_2)\searrow \mathcal{T}_3\\
   &=\sum_{s\in X_3}\sum_{r\in X_2}(\mathcal{T}_1\searrow_r \mathcal{T}_2)\searrow_s \mathcal{T}_3.
\end{align*}
Then
$$\mathcal{T}_1\searrow(\mathcal{T}_2\searrow \mathcal{T}_3)-(\mathcal{T}_1\searrow \mathcal{T}_2)\searrow \mathcal{T}_3=\sum_{u,\, v\in X_3}\mathcal{T}_1 \searrow_u(\mathcal{T}_2\searrow_v \mathcal{T}_3).$$
Which is symmetric on $\mathcal{T}_1$ and $\mathcal{T}_2$. Then we obtain:
$$\mathcal{T}_1\searrow(\mathcal{T}_2\searrow \mathcal{T}_3)-(\mathcal{T}_1\searrow \mathcal{T}_2)\searrow \mathcal{T}_3=\mathcal{T}_2\searrow(\mathcal{T}_1\searrow \mathcal{T}_3)-(\mathcal{T}_2\searrow \mathcal{T}_1)\searrow \mathcal{T}_3.$$
Consequently, $(\mathbb{V},\searrow)$ is a twisted pre-Lie algebra, thus yielding a pre-Lie algebra structure on $\overline{\mathcal{K}}(\mathbb{T})$.
\end{proof}
	We showed that $(\mathbb{V},\searrow)$ is a twisted pre-Lie algebra, so we consider the Hopf symmetric algebra $\mathcal{H}'=S(\mathbb{V})$ equipped with its usual unshuffling coproduct $\Delta_{unsh}$. We extend the product $\searrow$ to $\mathbb{T}$
by using Definition \ref{def.} and we define a product $\star$ on $\mathbb{T}$ by: For any pair $X_1, X_2$ of finite sets
\begin{eqnarray*}
		\star:\mathbb{T}_{X_1}\otimes\mathbb{T}_{X_2}&\longrightarrow& \mathbb{T}_{X_1\sqcup X_2}\\
		(\mathcal{T}_1, \mathcal{T}_2)&\longmapsto& \sum_{\mathcal{T}_1}\mathcal{T}^{(1)}_1(\mathcal{T}^{(2)}_1\searrow \mathcal{T}_2).
	\end{eqnarray*}
By construction, the space $(\mathcal{H}', \star, \Delta_{unsh})$ is a cocommutative twisted Hopf algebra.
\begin{remark}
The species of finite connected posets (i.e. finite connected $T_0$ topological spaces) is a twisted pre-Lie subalgebra of $(\mathbb{V}, \searrow)$, and the species of finite posets is a Hopf subalgebra of $\mathcal{H}'$.
\end{remark}
\begin{example}
\begin{align*}
  (\fcolorbox{white}{white}{
 \scalebox{0.7}{
  \begin{picture}(15,17) (216,-264)
    \SetWidth{0.0}
    \SetColor{Black}
    \Vertex(210,-260){2}
    \Vertex(226,-260){2}
    \Vertex(218,-249){2}
    \SetWidth{1.0}
    \Line(210,-261)(219,-248)
    \Line(226,-260)(218,-250)
    \Vertex(233,-260){2}
  \end{picture}
}})\searrow \fcolorbox{white}{white}{
\scalebox{0.7}{
  \begin{picture}(-6,17) (242,-291)
    \SetWidth{1.0}
    \SetColor{Black}
    \Line(235,-278)(235,-288)
    \SetWidth{0.0}
    \Vertex(235,-288){2}
    \SetWidth{1.0}
    \Vertex(235,-277){2}
  \end{picture}
}}&=\fcolorbox{white}{white}{
 \scalebox{0.7}{
  \begin{picture}(7,17) (216,-264)
    \SetWidth{0.0}
    \SetColor{Black}
    \Vertex(210,-260){2}
    \Vertex(226,-260){2}
    \Vertex(218,-249){2}
    \SetWidth{1.0}
    \Line(210,-261)(219,-248)
    \Line(226,-260)(218,-250)
  \end{picture}
}}\searrow (\fcolorbox{white}{white}{
\scalebox{0.7}{
  \begin{picture}(-10,7) (215,-286)
    \SetWidth{1.0}
    \SetColor{Black}
    \Vertex(208,-282){2}
  \end{picture}
}}\searrow  \fcolorbox{white}{white}{
\scalebox{0.7}{
  \begin{picture}(-6,17) (242,-291)
    \SetWidth{1.0}
    \SetColor{Black}
    \Line(235,-278)(235,-288)
    \SetWidth{0.0}
    \Vertex(235,-288){2}
    \SetWidth{1.0}
    \Vertex(235,-277){2}
  \end{picture}
}})-(\fcolorbox{white}{white}{
 \scalebox{0.7}{
  \begin{picture}(7,17) (216,-264)
    \SetWidth{0.0}
    \SetColor{Black}
    \Vertex(210,-260){2}
    \Vertex(226,-260){2}
    \Vertex(218,-249){2}
    \SetWidth{1.0}
    \Line(210,-261)(219,-248)
    \Line(226,-260)(218,-250)
  \end{picture}
}}\searrow \fcolorbox{white}{white}{
\scalebox{0.7}{
  \begin{picture}(-6,7) (215,-286)
    \SetWidth{1.0}
    \SetColor{Black}
    \Vertex(208,-282){2}
  \end{picture}
}})\searrow  \fcolorbox{white}{white}{
\scalebox{0.7}{
  \begin{picture}(-6,17) (242,-291)
    \SetWidth{1.0}
    \SetColor{Black}
    \Line(235,-278)(235,-288)
    \SetWidth{0.0}
    \Vertex(235,-288){2}
    \SetWidth{1.0}
    \Vertex(235,-277){2}
  \end{picture}
}}\\
&=\fcolorbox{white}{white}{
 \scalebox{0.7}{
  \begin{picture}(7,17) (216,-264)
    \SetWidth{0.0}
    \SetColor{Black}
    \Vertex(210,-260){2}
    \Vertex(226,-260){2}
    \Vertex(218,-249){2}
    \SetWidth{1.0}
    \Line(210,-261)(219,-248)
    \Line(226,-260)(218,-250)
  \end{picture}
}}\searrow ( \fcolorbox{white}{white}{
 \scalebox{0.7}{
  \begin{picture}(4,30) (285,-276)
    \SetWidth{0.0}
    \SetColor{Black}
    \Vertex(284,-261){2}
    \Vertex(284,-274){2}
    \SetWidth{1.0}
    \Line(284,-274)(284,-261)
    \Vertex(284,-249){2}
    \Line(284,-261)(284,-249)
  \end{picture}
}}+ \fcolorbox{white}{white}{
\scalebox{0.7}{
  \begin{picture}(13,19) (217,-274)
    \SetWidth{1.0}
    \SetColor{Black}
    \Vertex(219,-270){2}
    \Vertex(210,-258){2}
    \Vertex(226,-258){2}
    \Line(218,-271)(226,-257)
    \Line(220,-270)(209,-258)
  \end{picture}
}})- \fcolorbox{white}{white}{
\scalebox{0.7}{
  \begin{picture}(12,30) (281,-265)
    \SetWidth{0.0}
    \SetColor{Black}
    \Vertex(284,-261){2}
    \Vertex(293,-248){2}
    \Vertex(274,-248){2}
    \Vertex(284,-237){2}
    \SetWidth{1.0}
    \Line(284,-261)(292,-249)
    \Line(284,-261)(274,-248)
    \Line(275,-248)(284,-237)
    \Line(293,-248)(284,-237)
  \end{picture}
}}\searrow \fcolorbox{white}{white}{
\scalebox{0.7}{
  \begin{picture}(-6,17) (242,-291)
    \SetWidth{1.0}
    \SetColor{Black}
    \Line(235,-278)(235,-288)
    \SetWidth{0.0}
    \Vertex(235,-288){2}
    \SetWidth{1.0}
    \Vertex(235,-277){2}
  \end{picture}
}}\\
&=\fcolorbox{white}{white}{
\scalebox{0.7}{
  \begin{picture}(18,56) (278,-265)
    \SetWidth{0.0}
    \SetColor{Black}
    \Vertex(284,-248){2}
    \Vertex(284,-235){2}
    \SetWidth{1.0}
    \Line(284,-248)(284,-235)
    \SetWidth{0.0}
    \Vertex(293,-222){2}
    \Vertex(274,-222){2}
    \Vertex(284,-211){2}
    \SetWidth{1.0}
    \Line(284,-235)(292,-223)
    \Line(284,-235)(274,-222)
    \Line(275,-222)(284,-211)
    \Line(293,-222)(284,-211)
    \Vertex(284,-261){2}
    \Line(284,-261)(284,-248)
  \end{picture}
}}+  \fcolorbox{white}{white}{
\scalebox{0.7}{
  \begin{picture}(26,43) (279,-265)
    \SetWidth{1.0}
    \SetColor{Black}
    \Vertex(284,-261){2}
    \Vertex(284,-248){2}
    \Line(284,-261)(284,-248)
    \Vertex(293,-235){2}
    \Vertex(274,-235){2}
    \Vertex(284,-224){2}
    \Line(284,-248)(292,-236)
    \Line(284,-248)(274,-235)
    \Line(275,-235)(284,-224)
    \Line(293,-235)(284,-224)
    \Vertex(303,-235){2}
    \Line(285,-248)(303,-235)
  \end{picture}
}}+\fcolorbox{white}{white}{
\scalebox{0.7}{
  \begin{picture}(30,31) (276,-264)
    \SetWidth{0.0}
    \SetColor{Black}
    \Vertex(284,-260){2}
    \Vertex(293,-247){2}
    \Vertex(274,-247){2}
    \Vertex(284,-236){2}
    \SetWidth{1.0}
    \Line(284,-260)(292,-248)
    \Line(284,-260)(274,-247)
    \Line(275,-247)(284,-236)
    \Line(293,-247)(284,-236)
    \Vertex(303,-247){2}
    \Vertex(303,-236){2}
    \Line(285,-260)(303,-247)
    \Line(303,-247)(303,-237)
  \end{picture}
}} +\fcolorbox{white}{white}{
\scalebox{0.7}{
  \begin{picture}(44,30) (264,-265)
    \SetWidth{1.0}
    \SetColor{Black}
    \Vertex(284,-261){2}
    \Vertex(293,-248){2}
    \Vertex(274,-248){2}
    \Vertex(284,-237){2}
    \Line(284,-261)(292,-249)
    \Line(284,-261)(274,-248)
    \Line(275,-248)(284,-237)
    \Line(293,-248)(284,-237)
    \Vertex(303,-248){2.236}
    \Line(285,-261)(303,-248)
    \Vertex(263,-248){2}
    \Line(284,-262)(264,-249)
  \end{picture}
}} +2 \fcolorbox{white}{white}{
\scalebox{0.7}{
  \begin{picture}(21,43) (278,-265)
    \SetWidth{0.0}
    \SetColor{Black}
    \Vertex(284,-261){2}
    \Vertex(284,-248){2}
    \SetWidth{1.0}
    \Line(284,-261)(284,-248)
    \SetWidth{0.0}
    \Vertex(293,-235){2}
    \Vertex(274,-235){2}
    \Vertex(284,-224){2}
    \SetWidth{1.0}
    \Line(284,-248)(292,-236)
    \Line(284,-248)(274,-235)
    \Line(275,-235)(284,-224)
    \Line(293,-235)(284,-224)
    \Vertex(294,-249){2}
    \Line(284,-261)(294,-249)
  \end{picture}
}}\\
&\hspace{1cm}-\fcolorbox{white}{white}{
\scalebox{0.7}{
  \begin{picture}(16,56) (278,-265)
    \SetWidth{0.0}
    \SetColor{Black}
    \Vertex(284,-248){2}
    \Vertex(284,-235){2}
    \SetWidth{1.0}
    \Line(284,-248)(284,-235)
    \SetWidth{0.0}
    \Vertex(293,-222){2}
    \Vertex(274,-222){2}
    \Vertex(284,-211){2}
    \SetWidth{1.0}
    \Line(284,-235)(292,-223)
    \Line(284,-235)(274,-222)
    \Line(275,-222)(284,-211)
    \Line(293,-222)(284,-211)
    \Vertex(284,-261){2}
    \Line(284,-261)(284,-248)
  \end{picture}
}}-\fcolorbox{white}{white}{
\scalebox{0.7}{
  \begin{picture}(21,43) (278,-265)
    \SetWidth{0.0}
    \SetColor{Black}
    \Vertex(284,-261){2}
    \Vertex(284,-248){2}
    \SetWidth{1.0}
    \Line(284,-261)(284,-248)
    \SetWidth{0.0}
    \Vertex(293,-235){2}
    \Vertex(274,-235){2}
    \Vertex(284,-224){2}
    \SetWidth{1.0}
    \Line(284,-248)(292,-236)
    \Line(284,-248)(274,-235)
    \Line(275,-235)(284,-224)
    \Line(293,-235)(284,-224)
    \Vertex(294,-249){2}
    \Line(284,-261)(294,-249)
  \end{picture}
}}\\
&= \fcolorbox{white}{white}{
\scalebox{0.7}{
  \begin{picture}(26,43) (279,-265)
    \SetWidth{1.0}
    \SetColor{Black}
    \Vertex(284,-261){2}
    \Vertex(284,-248){2}
    \Line(284,-261)(284,-248)
    \Vertex(293,-235){2}
    \Vertex(274,-235){2}
    \Vertex(284,-224){2}
    \Line(284,-248)(292,-236)
    \Line(284,-248)(274,-235)
    \Line(275,-235)(284,-224)
    \Line(293,-235)(284,-224)
    \Vertex(303,-235){2}
    \Line(285,-248)(303,-235)
  \end{picture}
}}+\fcolorbox{white}{white}{
\scalebox{0.7}{
  \begin{picture}(30,31) (276,-264)
    \SetWidth{0.0}
    \SetColor{Black}
    \Vertex(284,-260){2}
    \Vertex(293,-247){2}
    \Vertex(274,-247){2}
    \Vertex(284,-236){2}
    \SetWidth{1.0}
    \Line(284,-260)(292,-248)
    \Line(284,-260)(274,-247)
    \Line(275,-247)(284,-236)
    \Line(293,-247)(284,-236)
    \Vertex(303,-247){2}
    \Vertex(303,-236){2}
    \Line(285,-260)(303,-247)
    \Line(303,-247)(303,-237)
  \end{picture}
}} +\fcolorbox{white}{white}{
\scalebox{0.7}{
  \begin{picture}(44,30) (264,-265)
    \SetWidth{1.0}
    \SetColor{Black}
    \Vertex(284,-261){2}
    \Vertex(293,-248){2}
    \Vertex(274,-248){2}
    \Vertex(284,-237){2}
    \Line(284,-261)(292,-249)
    \Line(284,-261)(274,-248)
    \Line(275,-248)(284,-237)
    \Line(293,-248)(284,-237)
    \Vertex(303,-248){2.236}
    \Line(285,-261)(303,-248)
    \Vertex(263,-248){2}
    \Line(284,-262)(264,-249)
  \end{picture}
}} +\fcolorbox{white}{white}{
\scalebox{0.7}{
  \begin{picture}(21,43) (278,-265)
    \SetWidth{0.0}
    \SetColor{Black}
    \Vertex(284,-261){2}
    \Vertex(284,-248){2}
    \SetWidth{1.0}
    \Line(284,-261)(284,-248)
    \SetWidth{0.0}
    \Vertex(293,-235){2}
    \Vertex(274,-235){2}
    \Vertex(284,-224){2}
    \SetWidth{1.0}
    \Line(284,-248)(292,-236)
    \Line(284,-248)(274,-235)
    \Line(275,-235)(284,-224)
    \Line(293,-235)(284,-224)
    \Vertex(294,-249){2}
    \Line(284,-261)(294,-249)
  \end{picture}
}}
\end{align*}

\begin{align*}
   (\fcolorbox{white}{white}{
 \scalebox{0.7}{
  \begin{picture}(15,17) (216,-264)
    \SetWidth{0.0}
    \SetColor{Black}
    \Vertex(210,-260){2}
    \Vertex(226,-260){2}
    \Vertex(218,-249){2}
    \SetWidth{1.0}
    \Line(210,-261)(219,-248)
    \Line(226,-260)(218,-250)
    \Vertex(233,-260){2}
  \end{picture}
}})\star \fcolorbox{white}{white}{
\scalebox{0.7}{
  \begin{picture}(-12,17) (243,-291)
    \SetWidth{1.0}
    \SetColor{Black}
    \Line(235,-278)(235,-288)
    \SetWidth{0.0}
    \Vertex(235,-288){2}
    \SetWidth{1.0}
    \Vertex(235,-277){2}
  \end{picture}
}}&=(\fcolorbox{white}{white}{
 \scalebox{0.7}{
  \begin{picture}(15,17) (216,-264)
    \SetWidth{0.0}
    \SetColor{Black}
    \Vertex(210,-260){2}
    \Vertex(226,-260){2}
    \Vertex(218,-249){2}
    \SetWidth{1.0}
    \Line(210,-261)(219,-248)
    \Line(226,-260)(218,-250)
    \Vertex(233,-260){2}
  \end{picture}
}})\searrow \fcolorbox{white}{white}{
\scalebox{0.7}{
  \begin{picture}(-6,17) (242,-291)
    \SetWidth{1.0}
    \SetColor{Black}
    \Line(235,-278)(235,-288)
    \SetWidth{0.0}
    \Vertex(235,-288){2}
    \SetWidth{1.0}
    \Vertex(235,-277){2}
  \end{picture}
}}+\fcolorbox{white}{white}{
 \scalebox{0.7}{
  \begin{picture}(23,17) (216,-264)
    \SetWidth{0.0}
    \SetColor{Black}
    \Vertex(210,-260){2}
    \Vertex(226,-260){2}
    \Vertex(218,-249){2}
    \SetWidth{1.0}
    \Line(210,-261)(219,-248)
    \Line(226,-260)(218,-250)
    \Vertex(233,-260){2}
        \Line(240,-249)(240,-260)
    \SetWidth{0.0}
    \Vertex(240,-260){2}
    \SetWidth{1.0}
    \Vertex(240,-249){2}
  \end{picture}
}} +\fcolorbox{white}{white}{
 \scalebox{0.7}{
  \begin{picture}(10,17) (216,-264)
    \SetWidth{0.0}
    \SetColor{Black}
    \Vertex(210,-260){2}
    \Vertex(226,-260){2}
    \Vertex(218,-249){2}
    \SetWidth{1.0}
    \Line(210,-261)(219,-248)
    \Line(226,-260)(218,-250)
  \end{picture}
}}(\fcolorbox{white}{white}{
\scalebox{0.7}{
  \begin{picture}(-12,7) (218,-286)
    \SetWidth{1.0}
    \SetColor{Black}
    \Vertex(208,-282){2}
  \end{picture}
}} \searrow \fcolorbox{white}{white}{
\scalebox{0.7}{
  \begin{picture}(-3,17) (242,-291)
    \SetWidth{1.0}
    \SetColor{Black}
    \Line(235,-278)(235,-288)
    \SetWidth{0.0}
    \Vertex(235,-288){2}
    \SetWidth{1.0}
    \Vertex(235,-277){2}
  \end{picture}
}})+ \fcolorbox{white}{white}{
\scalebox{0.7}{
  \begin{picture}(-6,7) (215,-286)
    \SetWidth{1.0}
    \SetColor{Black}
    \Vertex(208,-282){2}
  \end{picture}
}} (\fcolorbox{white}{white}{
 \scalebox{0.7}{
  \begin{picture}(10,17) (213,-264)
    \SetWidth{0.0}
    \SetColor{Black}
    \Vertex(210,-260){2}
    \Vertex(226,-260){2}
    \Vertex(218,-249){2}
    \SetWidth{1.0}
    \Line(210,-261)(219,-248)
    \Line(226,-260)(218,-250)
  \end{picture}
}} \searrow \fcolorbox{white}{white}{
\scalebox{0.7}{
  \begin{picture}(-2,17) (241,-291)
    \SetWidth{1.0}
    \SetColor{Black}
    \Line(235,-278)(235,-288)
    \SetWidth{0.0}
    \Vertex(235,-288){2}
    \SetWidth{1.0}
    \Vertex(235,-277){2}
  \end{picture}
}} )\\ 
&= \fcolorbox{white}{white}{
\scalebox{0.7}{
  \begin{picture}(26,43) (279,-265)
    \SetWidth{1.0}
    \SetColor{Black}
    \Vertex(284,-261){2}
    \Vertex(284,-248){2}
    \Line(284,-261)(284,-248)
    \Vertex(293,-235){2}
    \Vertex(274,-235){2}
    \Vertex(284,-224){2}
    \Line(284,-248)(292,-236)
    \Line(284,-248)(274,-235)
    \Line(275,-235)(284,-224)
    \Line(293,-235)(284,-224)
    \Vertex(303,-235){2}
    \Line(285,-248)(303,-235)
  \end{picture}
}}+ \fcolorbox{white}{white}{
\scalebox{0.7}{
  \begin{picture}(44,30) (264,-265)
    \SetWidth{1.0}
    \SetColor{Black}
    \Vertex(284,-261){2}
    \Vertex(293,-248){2}
    \Vertex(274,-248){2}
    \Vertex(284,-237){2}
    \Line(284,-261)(292,-249)
    \Line(284,-261)(274,-248)
    \Line(275,-248)(284,-237)
    \Line(293,-248)(284,-237)
    \Vertex(303,-248){2.236}
    \Line(285,-261)(303,-248)
    \Vertex(263,-248){2}
    \Line(284,-262)(264,-249)
  \end{picture}
}} + \fcolorbox{white}{white}{
\scalebox{0.7}{
  \begin{picture}(21,43) (278,-265)
    \SetWidth{0.0}
    \SetColor{Black}
    \Vertex(284,-261){2}
    \Vertex(284,-248){2}
    \SetWidth{1.0}
    \Line(284,-261)(284,-248)
    \SetWidth{0.0}
    \Vertex(293,-235){2}
    \Vertex(274,-235){2}
    \Vertex(284,-224){2}
    \SetWidth{1.0}
    \Line(284,-248)(292,-236)
    \Line(284,-248)(274,-235)
    \Line(275,-235)(284,-224)
    \Line(293,-235)(284,-224)
    \Vertex(294,-249){2}
    \Line(284,-261)(294,-249)
  \end{picture}
}} + \fcolorbox{white}{white}{
\scalebox{0.7}{
  \begin{picture}(30,31) (276,-264)
    \SetWidth{0.0}
    \SetColor{Black}
    \Vertex(284,-260){2}
    \Vertex(293,-247){2}
    \Vertex(274,-247){2}
    \Vertex(284,-236){2}
    \SetWidth{1.0}
    \Line(284,-260)(292,-248)
    \Line(284,-260)(274,-247)
    \Line(275,-247)(284,-236)
    \Line(293,-247)(284,-236)
    \Vertex(303,-247){2}
    \Vertex(303,-236){2}
    \Line(285,-260)(303,-247)
    \Line(303,-247)(303,-237)
  \end{picture}
}} + \fcolorbox{white}{white}{
 \scalebox{0.7}{
  \begin{picture}(23,17) (216,-264)
    \SetWidth{0.0}
    \SetColor{Black}
    \Vertex(210,-260){2}
    \Vertex(226,-260){2}
    \Vertex(218,-249){2}
    \SetWidth{1.0}
    \Line(210,-261)(219,-248)
    \Line(226,-260)(218,-250)
    \Vertex(233,-260){2}
        \Line(240,-249)(240,-260)
    \SetWidth{0.0}
    \Vertex(240,-260){2}
    \SetWidth{1.0}
    \Vertex(240,-249){2}
  \end{picture}
}}+  \fcolorbox{white}{white}{
 \scalebox{0.7}{
  \begin{picture}(15,17) (216,-264)
    \SetWidth{0.0}
    \SetColor{Black}
    \Vertex(210,-260){2}
    \Vertex(226,-260){2}
    \Vertex(218,-249){2}
    \SetWidth{1.0}
    \Line(210,-261)(219,-248)
    \Line(226,-260)(218,-250)
  \end{picture}
}} ( \fcolorbox{white}{white}{
 \scalebox{0.7}{
  \begin{picture}(4,30) (285,-276)
    \SetWidth{0.0}
    \SetColor{Black}
    \Vertex(284,-261){2}
    \Vertex(284,-274){2}
    \SetWidth{1.0}
    \Line(284,-274)(284,-261)
    \Vertex(284,-249){2}
    \Line(284,-261)(284,-249)
  \end{picture}
}}+ \fcolorbox{white}{white}{
\scalebox{0.7}{
  \begin{picture}(13,19) (217,-274)
    \SetWidth{1.0}
    \SetColor{Black}
    \Vertex(219,-270){2}
    \Vertex(210,-258){2}
    \Vertex(226,-258){2}
    \Line(218,-271)(226,-257)
    \Line(220,-270)(209,-258)
  \end{picture}
}})\\
&\hspace{1cm}+ \fcolorbox{white}{white}{
\scalebox{0.7}{
  \begin{picture}(-9,7) (218,-286)
    \SetWidth{1.0}
    \SetColor{Black}
    \Vertex(208,-282){2}
  \end{picture}
}} ( \fcolorbox{white}{white}{
\scalebox{0.7}{
  \begin{picture}(33,30) (275,-265)
    \SetWidth{0.0}
    \SetColor{Black}
    \Vertex(284,-261){2}
    \Vertex(293,-248){2}
    \Vertex(274,-248){2}
    \Vertex(284,-237){2}
    \SetWidth{1.0}
    \Line(284,-261)(292,-249)
    \Line(284,-261)(274,-248)
    \Line(275,-248)(284,-237)
    \Line(293,-248)(284,-237)
    \Vertex(303,-248){2}
    \Line(285,-261)(303,-248)
  \end{picture}
}}+ \fcolorbox{white}{white}{
\scalebox{0.7}{
  \begin{picture}(23,41) (277,-265)
    \SetWidth{0.0}
    \SetColor{Black}
    \Vertex(284,-250){2}
    \Vertex(293,-237){2}
    \Vertex(274,-237){2}
    \Vertex(284,-226){2}
    \SetWidth{1.0}
    \Line(284,-250)(292,-238)
    \Line(284,-250)(274,-237)
    \Line(275,-237)(284,-226)
    \Line(293,-237)(284,-226)
    \Vertex(284,-263){2}
    \Line(284,-263)(284,-250)
  \end{picture}
}})\\
&= \fcolorbox{white}{white}{
\scalebox{0.7}{
  \begin{picture}(26,43) (279,-265)
    \SetWidth{1.0}
    \SetColor{Black}
    \Vertex(284,-261){2}
    \Vertex(284,-248){2}
    \Line(284,-261)(284,-248)
    \Vertex(293,-235){2}
    \Vertex(274,-235){2}
    \Vertex(284,-224){2}
    \Line(284,-248)(292,-236)
    \Line(284,-248)(274,-235)
    \Line(275,-235)(284,-224)
    \Line(293,-235)(284,-224)
    \Vertex(303,-235){2}
    \Line(285,-248)(303,-235)
  \end{picture}
}}+ \fcolorbox{white}{white}{
\scalebox{0.7}{
  \begin{picture}(44,30) (264,-265)
    \SetWidth{1.0}
    \SetColor{Black}
    \Vertex(284,-261){2}
    \Vertex(293,-248){2}
    \Vertex(274,-248){2}
    \Vertex(284,-237){2}
    \Line(284,-261)(292,-249)
    \Line(284,-261)(274,-248)
    \Line(275,-248)(284,-237)
    \Line(293,-248)(284,-237)
    \Vertex(303,-248){2.236}
    \Line(285,-261)(303,-248)
    \Vertex(263,-248){2}
    \Line(284,-262)(264,-249)
  \end{picture}
}} + \fcolorbox{white}{white}{
\scalebox{0.7}{
  \begin{picture}(21,43) (278,-265)
    \SetWidth{0.0}
    \SetColor{Black}
    \Vertex(284,-261){2}
    \Vertex(284,-248){2}
    \SetWidth{1.0}
    \Line(284,-261)(284,-248)
    \SetWidth{0.0}
    \Vertex(293,-235){2}
    \Vertex(274,-235){2}
    \Vertex(284,-224){2}
    \SetWidth{1.0}
    \Line(284,-248)(292,-236)
    \Line(284,-248)(274,-235)
    \Line(275,-235)(284,-224)
    \Line(293,-235)(284,-224)
    \Vertex(294,-249){2}
    \Line(284,-261)(294,-249)
  \end{picture}
}} + \fcolorbox{white}{white}{
\scalebox{0.7}{
  \begin{picture}(30,31) (276,-264)
    \SetWidth{0.0}
    \SetColor{Black}
    \Vertex(284,-260){2}
    \Vertex(293,-247){2}
    \Vertex(274,-247){2}
    \Vertex(284,-236){2}
    \SetWidth{1.0}
    \Line(284,-260)(292,-248)
    \Line(284,-260)(274,-247)
    \Line(275,-247)(284,-236)
    \Line(293,-247)(284,-236)
    \Vertex(303,-247){2}
    \Vertex(303,-236){2}
    \Line(285,-260)(303,-247)
    \Line(303,-247)(303,-237)
  \end{picture}
}} + \fcolorbox{white}{white}{
 \scalebox{0.7}{
  \begin{picture}(23,17) (216,-264)
    \SetWidth{0.0}
    \SetColor{Black}
    \Vertex(210,-260){2}
    \Vertex(226,-260){2}
    \Vertex(218,-249){2}
    \SetWidth{1.0}
    \Line(210,-261)(219,-248)
    \Line(226,-260)(218,-250)
    \Vertex(233,-260){2}
        \Line(240,-249)(240,-260)
    \SetWidth{0.0}
    \Vertex(240,-260){2}
    \SetWidth{1.0}
    \Vertex(240,-249){2}
  \end{picture}
}}+\fcolorbox{white}{white}{
 \scalebox{0.7}{
  \begin{picture}(10,17) (216,-264)
    \SetWidth{0.0}
    \SetColor{Black}
    \Vertex(210,-260){2}
    \Vertex(226,-260){2}
    \Vertex(218,-249){2}
    \SetWidth{1.0}
    \Line(210,-261)(219,-248)
    \Line(226,-260)(218,-250)
  \end{picture}
}}  \fcolorbox{white}{white}{
 \scalebox{0.7}{
  \begin{picture}(-4,30) (292,-276)
    \SetWidth{0.0}
    \SetColor{Black}
    \Vertex(284,-261){2}
    \Vertex(284,-272){2}
    \SetWidth{1.0}
    \Line(284,-272)(284,-261)
    \Vertex(284,-249){2}
    \Line(284,-261)(284,-249)
  \end{picture}
}}+ \fcolorbox{white}{white}{
 \scalebox{0.7}{
  \begin{picture}(8,17) (216,-264)
    \SetWidth{0.0}
    \SetColor{Black}
    \Vertex(210,-260){2}
    \Vertex(226,-260){2}
    \Vertex(218,-249){2}
    \SetWidth{1.0}
    \Line(210,-261)(219,-248)
    \Line(226,-260)(218,-250)
  \end{picture}
}} \fcolorbox{white}{white}{
\scalebox{0.7}{
  \begin{picture}(13,19) (220,-274)
    \SetWidth{1.0}
    \SetColor{Black}
    \Vertex(219,-270){2}
    \Vertex(210,-258){2}
    \Vertex(226,-258){2}
    \Line(218,-271)(226,-257)
    \Line(220,-270)(209,-258)
  \end{picture}
}}\\
&\hspace{1cm}+  \fcolorbox{white}{white}{
\scalebox{0.7}{
  \begin{picture}(33,30) (270,-265)
    \SetWidth{0.0}
    \SetColor{Black}
    \Vertex(284,-261){2}
    \Vertex(293,-248){2}
    \Vertex(274,-248){2}
    \Vertex(284,-237){2}
    \SetWidth{1.0}
    \Line(284,-261)(292,-249)
    \Line(284,-261)(274,-248)
    \Line(275,-248)(284,-237)
    \Line(293,-248)(284,-237)
    \Vertex(303,-248){2}
    \Line(285,-261)(303,-248)
    \Vertex(264,-261){2}
  \end{picture}
}}+ \fcolorbox{white}{white}{
\scalebox{0.7}{
  \begin{picture}(23,41) (268,-265)
    \SetWidth{0.0}
    \SetColor{Black}
    \Vertex(284,-250){2}
    \Vertex(293,-237){2}
    \Vertex(274,-237){2}
    \Vertex(284,-226){2}
    \SetWidth{1.0}
    \Line(284,-250)(292,-238)
    \Line(284,-250)(274,-237)
    \Line(275,-237)(284,-226)
    \Line(293,-237)(284,-226)
    \Vertex(284,-263){2}
    \Line(284,-263)(284,-250)
    \Vertex(267,-263){2}
  \end{picture}
}}
\end{align*}
\end{example}
\section{Bialgebras of finite topological spaces}\label{bialgebras}
\subsection{ A twisted bialgebra of finite topological spaces}
Let $X$ be any finite set, we define the coproduct $\Delta_{\searrow}$ by:
	\begin{eqnarray*}
		\Delta_{\searrow}:\mathbb{T}_X&\longrightarrow& (\mathbb{T}\otimes \mathbb{T})_X=\bigoplus_{Y\sqcup Z= X}\mathbb{T}_{Y}\otimes\mathbb{T}_{Z}\\
		\mathcal{T}&\longmapsto& \sum_{Y \overline{\in}  \mathcal{T}} \mathcal{T}_{|Y} \otimes \mathcal{T}_{|X\backslash Y}.
	\end{eqnarray*}
	Where $Y \overline{\in}  \mathcal{T}$, stands for
\begin{itemize}
    \item $Y\in \mathcal{T}$,
    \item $\mathcal{T}_{|Y}=\mathcal{T}_1...\mathcal{T}_n$, such that for all $i\in \{1,..., n\}, \mathcal{T}_i$ connected and \big($\hbox{min}\mathcal{T}_i=(\hbox{min}\mathcal{T})\cap \mathcal{T}_i$, or there is a single common ancestor $x_i \in \overline{X\backslash Y}$ to $\hbox{min}\mathcal{T}_i$\big), where $\overline{X\backslash Y}=(X\backslash Y)/\sim_{\mathcal{T}_{|X\backslash Y}}$.
\end{itemize}
\begin{example}
 
 $\Delta_{\searrow}(\fcolorbox{white}{white}{
 \scalebox{0.7}{
  \begin{picture}(8,17) (217,-263)
    \SetWidth{1.0}
    \SetColor{Black}
    \Vertex(210,-260){2}
    \Vertex(226,-260){2}
    \Vertex(218,-249){2}
    \Line(210,-261)(219,-248)
    \Line(226,-260)(218,-250)
  \end{picture}
}})=\hspace*{-0.09cm} \fcolorbox{white}{white}{
\scalebox{0.7}{
  \begin{picture}(8,17) (218,-263)
    \SetWidth{1.0}
    \SetColor{Black}
    \Vertex(210,-260){2}
    \Vertex(226,-260){2}
    \Vertex(218,-249){2}
    \Line(210,-261)(219,-248)
    \Line(226,-260)(218,-250)
  \end{picture}
}}\hspace*{-0.09cm} \otimes \mathbf{1}+\mathbf{1}\otimes \hspace*{-0.05cm} \fcolorbox{white}{white}{
\scalebox{0.7}{
  \begin{picture}(22,17) (218,-263)
    \SetWidth{1.0}
    \SetColor{Black}
    \Vertex(210,-260){2}
    \Vertex(226,-260){2}
    \Vertex(218,-249){2}
    \Line(210,-261)(219,-248)
    \Line(226,-260)(218,-250)
  \end{picture}
}}$\\
\hspace*{2.57cm}$\Delta_{\searrow}(\fcolorbox{white}{white}{
\scalebox{0.7}{
  \begin{picture}(8,19) (218,-274)
    \SetWidth{1.0}
    \SetColor{Black}
    \Vertex(219,-270){2}
    \Vertex(210,-258){2}
    \Vertex(226,-258){2}
    \Line(218,-271)(226,-257)
    \Line(220,-270)(209,-258)
  \end{picture}
}})=\fcolorbox{white}{white}{
\scalebox{0.7}{
  \begin{picture}(-6,7) (216,-286)
    \SetWidth{1.0}
    \SetColor{Black}
    \Vertex(208,-282){2}
    \Vertex(214,-282){2}
  \end{picture}
}}\otimes \fcolorbox{white}{white}{
\scalebox{0.7}{
  \begin{picture}(-9,7) (218,-286)
    \SetWidth{1.0}
    \SetColor{Black}
    \Vertex(208,-282){2}
  \end{picture}
}}+ \fcolorbox{white}{white}{
\scalebox{0.7}{
  \begin{picture}(5,19) (218,-274)
    \SetWidth{1.0}
    \SetColor{Black}
    \Vertex(219,-270){2}
    \Vertex(210,-258){2}
    \Vertex(226,-258){2}
    \Line(218,-271)(226,-257)
    \Line(220,-270)(209,-258)
  \end{picture}
}} \otimes \mathbf{1}+\mathbf{1}\otimes  \fcolorbox{white}{white}{
\scalebox{0.7}{
  \begin{picture}(27,19) (218,-274)
    \SetWidth{1.0}
    \SetColor{Black}
    \Vertex(219,-270){2}
    \Vertex(210,-258){2}
    \Vertex(226,-258){2}
    \Line(218,-271)(226,-257)
    \Line(220,-270)(209,-258)
  \end{picture}
}} $
\end{example}
\begin{theorem}
		$(\mathbb{T}, m, \Delta_{\searrow})$ is a commutative connected twisted bialgebra, and $\mathcal{H}=\overline{\mathcal{K}}(\mathbb{T})$ is a commutative graded bialgebra.
	\end{theorem}
	\begin{proof}
		To show that $\mathbb{T}$ is a twisted bialgebra \cite{acg10}, it is necessary to show that $\Delta_{\searrow}$ is coassociative, and that the species coproduct $\Delta_{\searrow}$ and the product defined by:
			\begin{align*}
	m:\mathbb{T}_{X_1}\otimes \mathbb{T}_{X_2} \longrightarrow \mathbb{T}_{X_1\sqcup X_2} \\
	\mathcal{T}_1\otimes \mathcal{T}_2\longmapsto \mathcal{T}_1\mathcal{T}_2,
	\end{align*} 
		are compatible.
		The unit $\mathbf 1$ is identified to the empty topology.
		Coassociativity is checked by a careful, but straighforward computation. We have
		\begin{align*}
		(\Delta_{\searrow} \otimes id)\Delta_{\searrow} (\mathcal{T})&=(\Delta_{\searrow} \otimes id)\left( \sum_{Y \overline{\in}  \mathcal{T}} \mathcal{T}_{|Y} \otimes \mathcal{T}_{|X\backslash Y}\right) \\
		&=\sum_{Z\overline{\in}\mathcal{T}_{|Y},\, Y \overline{\in}  \mathcal{T}} \mathcal{T}_{|Z}\otimes \mathcal{T}_{|Y\backslash Z} \otimes \mathcal{T}_{|X\backslash Y}.
		\end{align*}
	On the other hand
		\begin{align*}
		(id \otimes \Delta_{\searrow} )\Delta_{\searrow} (\mathcal{T})&=(id \otimes \Delta_{\searrow})\left( \sum_{U \overline{\in}  \mathcal{T}} \mathcal{T}_{|U} \otimes \mathcal{T}_{|X\backslash U} \right) \\
		&=\sum_{W\overline{\in}\mathcal{T}_{|X\backslash U},\, U \overline{\in}  \mathcal{T}} \mathcal{T}_{|U} \otimes \mathcal{T}_{|W} \otimes \mathcal{T}_{|X\backslash (U\sqcup W)}. 
		\end{align*}
		Coassociativity will come from the fact that $(Z, Y )\longmapsto (Z, Y\backslash Z)$ is a bijection from the set of pairs $(Z,Y)$ with $Y\overline{\in} \mathcal{T}$ and $Z\overline{\in} \mathcal{T}_{|Y}$ and, onto the set of pairs $(U, W)$ with $U\overline{\in} \mathcal{T}$ and $W\overline{\in} \mathcal{T}_{|X\backslash U}$. The inverse map is given by $(U, W)\longmapsto (U, U\sqcup W)$.\\
		\\
		Let $A=\{(Z, Y), Y\overline{\in} \mathcal{T} \hbox{ and } Z\overline{\in} \mathcal{T}_{|Y} \}$, and $B=\{(U, W), U\overline{\in} \mathcal{T} \hbox{ and } W\overline{\in} \mathcal{T}_{|X\backslash U} \}$.\\
			We define
			\begin{equation*}
			   \begin{split}
			         f:A&  \longrightarrow  B\\
			        (Z, Y) & \longmapsto  (Z, Y\backslash Z)
			    \end{split}
			     \ \ \ \ \ \ \ \ \ \ \ \ \ \ \ \ \ \ \ \ \ \ \ \ \ \ \ \ \ \begin{split}
			         g: B & \longrightarrow A\\
			        (U, W) & \longmapsto (U, U \sqcup W)
			    \end{split}
			\end{equation*}
Let us prove that $f$ and $g$ are well defined.\\
Let $(Z, Y)\in A$, i.e\\
\hspace*{0,5cm}- $Y\in \mathcal{T}$ and $\mathcal{T}_{|Y}=\mathcal{T}_1...\mathcal{T}_n$, such that for all $i\in \{1,...,n\}$,
$\mathcal{T}_i$ connected component and
$\big(\hbox{min}\mathcal{T}_i=(\hbox{min}\mathcal{T})\cap \mathcal{T}_i$ or there is a unique common ancestor $x_i\in \overline{X\backslash Y}$  to $\hbox{min}\mathcal{T}_i\big)$,\\
and\\
\hspace*{0,5cm}- $Z\in \mathcal{T}_{|Y}$ and $\mathcal{T}_{|Z}=\mathcal{T}_{1|{Z}}...\mathcal{T}_{n|{Z}}=\mathcal{T}_{1,1}\mathcal{T}_{1,2}...\mathcal{T}_{1,i_1}\mathcal{T}_{2,1}\mathcal{T}_{2,2}...\mathcal{T}_{2,i_2}...\mathcal{T}_{n,1}\mathcal{T}_{n,2}...\mathcal{T}_{n,i_n}$, such that for all $i\in \{1,...,n\}, j\in \{i_1,...,i_n\}$,
$\mathcal{T}_{i, j}$ connected component and $\big(\hbox{min}\mathcal{T}_{i, j}=\hbox{min}\mathcal{T}_{|Y}\cap \mathcal{T}_{i, j}$ or there is a unique common ancestor $x_{i, j}\in \overline{Y\backslash Z}$ to $\hbox{min}\mathcal{T}_{i, j}\big)$.\\
Then we can visualise $\mathcal{T}$ by the graph illustrated below in figure \ref{figure1}:\\
\begin{figure}[h!]
\centering
\begin{minipage}{5in}
\fcolorbox{white}{white}{
\scalebox{0.45}{
  \begin{picture}(741,406) (30,-75)
    \SetWidth{3.0}
    \SetColor{Black}
    \Line(32,246)(752,246)
    \Line(32,103)(752,102)
    \Arc(128,278)(16,180,540)
    \Line[dash,dashsize=2](160,278)(192,278)
    \Arc(224,278)(16,180,540)
    \Arc(464,278)(16,180,540)
    \Arc(272,278)(16,180,540)
    \Arc(368,278)(16,180,540)
    \Arc(560,278)(16,180,540)
    \Arc(656,278)(16,180,540)
    \Arc(752,278)(16,180,540)
    \Line[dash,dashsize=2](304,278)(336,278)
    \Line[dash,dashsize=2](496,278)(528,278)
    \Line[dash,dashsize=2](688,278)(720,278)
    \Line[dash,dashsize=2](391,276)(439,276)
    \Arc(320,182)(31,180,540)
    \Arc(512,182)(31,180,540)
    \Arc(400,-26)(48,180,540)
    \Arc(176,182)(31,180,540)
    \SetWidth{2.0}
    \Line(128,262)(162,211)
    \Line(160,214)(136,265)
    \Line(273,261)(323,212)
    \Line(355,268)(323,213)
    \Line(368,262)(322,211)
    \Line(458,262)(497,209)
    \Line(470,263)(496,209)
    \Line(477,268)(497,210)
    \Line(556,262)(526,211)
    \Line(642,268)(430,10)
    \Line(653,262)(432,11)
    \Line(735,275)(436,5)
    \Line(740,266)(436,5)
    \Line(753,261)(436,4)
    \Line(159,156)(360,2)
    \Line(181,151)(360,2)
    \Line(195,156)(360,2)
    \Line(320,151)(382,17)
    \Line(484,168)(401,22)
    \Line(499,153)(401,21)
    \Text(32,278)[lb]{\huge{\Black{$Z$}}}
    \Text(32,134)[lb]{\huge{\Black{$Y\setminus Z$}}}
    \Text(32,-58)[lb]{\huge{\Black{$X\setminus Y$}}}
    \Line(209,270)(184,212)
    \Line(219,263)(183,212)
    \Line(227,262)(183,212)
    \Text(390,-33)[lb]{\huge{\Black{$X$}}}
    \SetWidth{2.0}
    \Line[dash,dashsize=2](389,182)(437,182)
    \Line[dash,dashsize=2](400,182)(448,182)
    \Text(168,171)[lb]{\LARGE{\Black{$Y_1$}}}
    \Text(312,171)[lb]{\LARGE{\Black{$Y_2$}}}
    \Text(504,171)[lb]{\LARGE{\Black{$Y_n$}}}
    \Text(116,274)[lb]{\Large{\Black{$Z_{11}$}}}
    \Text(212,274)[lb]{\Large{\Black{$Z_{1k_1}$}}}
    \Text(261,274)[lb]{\Large{\Black{$Z_{21}$}}}
    \Text(356,274)[lb]{\Large{\Black{$Z_{2k_2}$}}}
    \Text(454,274)[lb]{\Large{\Black{$Z_{n1}$}}}
    \Text(550,274)[lb]{\Large{\Black{$Z_{nk_n}$}}}
    \Text(647,274)[lb]{\Large{\Black{$Z_{01}$}}}
    \Text(744,274)[lb]{\Large{\Black{$Z_{0k_0}$}}}
  \end{picture}
}
}
\end{minipage}
\caption{}
\label{figure1}
\end{figure}
Graphically it is clear that $Z\in \mathcal{T}$ and $Y\backslash Z \in \mathcal{T}_{|X\backslash Z}$. Then $(Z, Y\backslash Z)\in B$.\\
Then $f$ is well defined.\\
\\
Let $(U, W)\in B$, i.e\\
\hspace*{0,5cm}- $U\in \mathcal{T}$ and  $\mathcal{T}_{|U}=\mathcal{T}_{|U_1}...\mathcal{T}_{|U_p}$, such that for all $i\in \{1,...,p\},$
$\mathcal{T}_{|U_i}$ connected component and $\big( \hbox{min}\mathcal{T}_{|U_i}=(\hbox{min}\mathcal{T})\cap \mathcal{T}_{|U_i}$ or there is a unique $x_i\in \overline{X\backslash U}$ common ancestor to $\hbox{min}\mathcal{T}_{|U_i}\big)$.\\
and\\
\hspace*{0,5cm}- $W\in \mathcal{T}_{|X\backslash U}$ and  $\mathcal{T}_{|W}=\mathcal{T}^{1}...\mathcal{T}^{q}$, such that for all $j\in \{1,...,q\}$,
$\mathcal{T}^{j}$ connected component and $\big(\hbox{min}\mathcal{T}^{j}=\hbox{min}\mathcal{T}_{|X\backslash U}\cap \mathcal{T}^{j}$ or there is a unique $x_j\in \overline{X\backslash (U\sqcup W)}$ common ancestor to $\hbox{min}\mathcal{T}^{j}\big)$.\\
For all $k\in\{1,...,q\}, \hbox{ we notice } U^{k}=\bigsqcup \limits_{\underset{}{0\le n\le p}} U_i, \hbox{ where } U_i \hbox{ verifies the existence of } x_i\in \overline{W_k}=\overline{\mathcal{V}(\mathcal{T}^{k})}$ common ancestor to $\hbox{min}\mathcal{T}_{|U_i},$ where $v\in \mathcal{V}(\mathcal{T}^{k})$ denotes that $v$ is a element of the topological space $\mathcal{T}^{k}$.\\
We notice $U^0=\bigsqcup \limits_{\underset{}{0\le n\le p}} U_i, \hbox{ where } U_i \hbox{ verifies the existence of a unique } x_i\in \overline{X\backslash (U\sqcup W)}$ common ancestor to $\hbox{min}\mathcal{T}_{|U_i}.$\\
We notice $W_0\subset W, \hbox{ where } W_0 \hbox{ verifie thet for all } x\in W_0, \hbox{ there is no } y \in U \hbox{ such that } x\leq_{\mathcal{T}} y.$\\
Then we can visualise $\mathcal{T}$ by the graph illustrated below in figure \ref{figure2} below:\\
\begin{figure}[h!]
\centering
\begin{minipage}{5in}
\fcolorbox{white}{white}{
\scalebox{0.7}{
  \begin{picture}(741,406) (30,-75)
    \SetWidth{1.5}
    \SetColor{Black}
    \Arc(56.409,133.972)(15.955,135,495)
    \Arc(124.101,133.972)(15.955,135,495)
    \Arc(169.228,133.972)(15.955,135,495)
    \Arc(214.356,133.972)(15.955,135,495)
    \Arc(304.611,133.972)(15.955,135,495)
    \Arc(90.255,235.509)(11.282,90,450)
    \Arc(135.383,235.509)(11.282,180,540)
    \Arc(169.228,235.509)(11.282,180,540)
    \Arc(214.356,235.509)(11.282,180,540)
    \Arc(304.611,235.509)(11.282,180,540)
    \Arc(349.739,235.509)(11.282,90,450)
    \Arc(383.584,235.509)(11.282,90,450)
    \Arc(439.994,235.509)(11.282,90,450)
    \Arc(282.047,9.872)(25.227,153,513)
    \Line(304.611,21.154)(439.994,224.228)
    \Line(293.329,32.435)(383.584,224.228)
    \Line(270.765,32.435)(225.638,122.691)
    \Line(270.765,32.435)(203.074,122.691)
    \Line(259.483,21.154)(180.51,122.691)
    \Line(259.483,21.154)(135.383,122.691)
    \Line(259.483,-1.41)(67.691,122.691)
    \Line(259.483,-1.41)(45.128,122.691)
    \Line(259.483,21.154)(112.819,122.691)
    \Line(293.329,145.254)(304.611,224.228)
    \Line(293.329,145.254)(349.739,224.228)
    \Line(157.946,145.254)(90.255,224.228)
    \Line(157.946,145.254)(101.537,235.509)
    \Line(203.074,145.254)(169.228,224.228)
    \Line(225.638,145.254)(214.356,224.228)
    \Line(180.51,145.254)(135.383,224.228)
    \Line(293.329,145.254)(293.329,235.509)
    \Line(293.329,32.435)(372.302,235.509)
    \Line(304.611,21.154)(428.712,235.509)
    \Line(225.638,145.254)(203.074,235.509)
    \Line(203.074,145.254)(180.51,235.509)
    \Line[dash,dashsize=1.41](78.973,133.972)(101.537,133.972)
    \Line[dash,dashsize=1.41](236.92,133.972)(282.047,133.972)
    \Line(33.846,190.382)(462.557,190.382)
    \Line(33.846,77.563)(462.557,77.563)
    \Text(282.047,9.872)[lb]{\Large{\Black{$X$}}}
    \Line[dash,dashsize=1.41](103.652,235.509)(119.87,235.509)
    \Line[dash,dashsize=1.41](182.626,236.92)(199.548,236.215)
    \Line[dash,dashsize=1.41](318.008,236.215)(335.636,236.92)
    \Line[dash,dashsize=1.41](396.276,235.509)(423.071,235.509)
    \Line[dash,dashsize=1.41](232.689,235.509)(289.098,235.509)
    \Text(111.409,260.894)[lb]{\Large{\Black{$U^1$}}}
    \Text(191.792,260.189)[lb]{\Large{\Black{$U^2$}}}
    \Text(326.47,258.073)[lb]{\Large{\Black{$U^q$}}}
    \Text(89.55,155.831)[lb]{\Large{\Black{$W_0$}}}
    \Text(410.379,257.368)[lb]{\Large{\Black{$U^0$}}}
    \Text(9.872,236.92)[lb]{\Large{\Black{$U$}}}
    \Text(9.872,133.267)[lb]{\Large{\Black{$W$}}}
    \Line(385.7,253.137)(444.93,253.137)
    \Line(90.96,254.548)(141.729,254.548)
    \Line(170.639,254.548)(224.933,254.548)
    \Line(303.906,254.548)(358.2,254.548)
    \Line(54.999,153.011)(131.857,152.306)
    \Text(162.882,130.972)[lb]{\Large{\Black{$W_1$}}}
    \Text(207.305,130.267)[lb]{\Large{\Black{$W_2$}}}
    \Text(296.855,130.562)[lb]{\Large{\Black{$W_q$}}}
    \Text(11.282,20.448)[lb]{\Large{\Black{$X \backslash (U \sqcup W)$}}}
    \Line(259.483,21.154)(130.447,119.165)
    \Line(203.074,145.959)(175.574,225.638)
    \Line(157.946,145.959)(97.306,227.048)
    \Line(304.611,21.859)(432.237,227.048)
    \Line(283.457,34.551)(303.906,118.46)
    \Line(270.765,33.141)(216.471,118.46)
    \Line(282.752,35.256)(296.15,120.575)
  \end{picture}
  }
}
\end{minipage}
\vskip -1.2cm
\caption{}
\label{figure2}
\end{figure}
Graphically it is clear that $U\in \mathcal{T}_{|W\sqcup U}$ and $W\sqcup U\in \mathcal{T}$. Then $(U, U\sqcup W)\in A$.\\
Then $g$ is well defined.\\
We have for all $(Z, Y)\in A$ then $(Z, Y\backslash Z)\in B$, and for all $(U, W)\in B$ then $(U, U\sqcup W)\in A$. Then $| A |= | B |$.\\
Let $(Z_1, Y_1), (Z_2, Y_2) \in A$ such that $f(Z_1, Y_1)=f(Z_2, Y_2)$, then $Z_1=Z_2$, then $Y_1\backslash Z_1=Y_2\backslash Z_1$,\\
then $f$ is injective, then $f$ is bijective.\\
In the same way we show that $g$ is bijective, and $g\circ f=f \circ g=Id$.\\
Then $\Delta$ is coassociative.\\
Finally, we show immediately that
		$$\Delta_{\searrow}(\mathcal{T}_1\mathcal{T}_2)=\Delta_{\searrow}(\mathcal{T}_1)\Delta_{\searrow}(\mathcal{T}_2).$$
\end{proof}
\begin{remark}
\noindent For any finite set $X$, let us recall from \cite{acg6} the internal coproduct $\Gamma$ on $\mathbb{T}_X$:
	\begin{equation}
	 \Gamma(\mathcal{T})=\sum \limits_{\underset{}{\mathcal{T}^{\prime}\soprec \mathcal{T}}}\mathcal{T}^{\prime}\otimes \mathcal{T}/ \mathcal{T}^{\prime}.
	 \end{equation}
	The sum runs over topologies $\mathcal{T}^{\prime}$ which are $\mathcal{T}$-admissible, i.e 
\begin{itemize}
	\item finer than $\mathcal{T}$,
	\item such that $\mathcal{T}^{\prime}_{|Y}=\mathcal{T}_{|Y}$ for any subset $Y\subset X$ connected for the topology $\mathcal{T}^{\prime}$,\
	\item such that for any $x, y \in X$,
	\begin{equation}
	 x \sim_{\mathcal{T}/ \mathcal{T}^{\prime}} y \iff x \sim_{\mathcal{T}^{\prime}/ \mathcal{T}^{\prime}} y.
	 \end{equation}
\end{itemize}
F. Fauvet, L. Foissy, and D. Manchon in \cite{acg6} show that $\Gamma$ and $\Delta$ are compatible.\\
On the other hand, we notice that $\Gamma$ and $\Delta_{\searrow}$ are not compatible. In fact:\\

$\Delta_{\searrow}(\fcolorbox{white}{white}{
\scalebox{0.7}{
  \begin{picture}(10,30) (216,-263)
    \SetWidth{1.0}
    \SetColor{Black}
    \Vertex(218,-259){2}
    \Vertex(210,-247){2}
    \Vertex(226,-247){2}
    \Vertex(218,-236){2}
    \Line(218,-260)(226,-246)
    \Line(219,-260)(209,-247)
    \Line(210,-248)(219,-235)
    \Line(226,-247)(218,-237)
  \end{picture}
}})=\fcolorbox{white}{white}{
\scalebox{0.7}{
  \begin{picture}(5,17) (218,-263)
    \SetWidth{1.0}
    \SetColor{Black}
    \Vertex(210,-260){2}
    \Vertex(226,-260){2}
    \Vertex(218,-249){2}
    \Line(210,-261)(219,-248)
    \Line(226,-260)(218,-250)
  \end{picture}
}}\otimes  \fcolorbox{white}{white}{
\scalebox{0.7}{
  \begin{picture}(-9,17) (218,-263)
    \SetWidth{1.0}
    \SetColor{Black}
    \Vertex(210,-259){2}
  \end{picture}
}} + \mathbf{1}\otimes \hspace*{-0.08cm} \fcolorbox{white}{white}{
\scalebox{0.7}{
  \begin{picture}(10,30) (217,-263)
    \SetWidth{1.0}
    \SetColor{Black}
    \Vertex(218,-259){2}
    \Vertex(210,-247){2}
    \Vertex(226,-247){2}
    \Vertex(218,-236){2}
    \Line(218,-260)(226,-246)
    \Line(219,-260)(209,-247)
    \Line(210,-248)(219,-235)
    \Line(226,-247)(218,-237)
  \end{picture}
}} + \fcolorbox{white}{white}{
\scalebox{0.7}{
  \begin{picture}(7,30) (218,-263)
    \SetWidth{1.0}
    \SetColor{Black}
    \Vertex(218,-259){2}
    \Vertex(210,-247){2}
    \Vertex(226,-247){2}
    \Vertex(218,-236){2}
    \Line(218,-260)(226,-246)
    \Line(219,-260)(209,-247)
    \Line(210,-248)(219,-235)
    \Line(226,-247)(218,-237)
  \end{picture}
}} \otimes \mathbf{1}$\\

$\Gamma(\fcolorbox{white}{white}{
\scalebox{0.7}{
  \begin{picture}(13,30) (216,-263)
    \SetWidth{1.0}
    \SetColor{Black}
    \Vertex(218,-259){2}
    \Vertex(210,-247){2}
    \Vertex(226,-247){2}
    \Vertex(218,-236){2}
    \Line(218,-260)(226,-246)
    \Line(219,-260)(209,-247)
    \Line(210,-248)(219,-235)
    \Line(226,-247)(218,-237)
  \end{picture}
}})=\fcolorbox{white}{white}{
\scalebox{0.7}{
  \begin{picture}(14,19) (218,-274)
    \SetWidth{1.0}
    \SetColor{Black}
    \Vertex(219,-270){2}
    \Vertex(210,-258){2}
    \Vertex(226,-258){2}
    \Line(218,-271)(226,-257)
    \Line(220,-270)(209,-258)
    \Vertex(232,-270){2}
  \end{picture}
}}\otimes \fcolorbox{white}{white}{
\scalebox{0.7}{
  \begin{picture}(8,31) (235,-264)
    \SetWidth{1.0}
    \SetColor{Black}
    \Vertex(232,-257){2}
    \Vertex(239,-257){2}
    \Vertex(236,-252){2}
    \Arc(235,-255)(8.062,187,547)
    \Vertex(235,-237){2}
    \Line(235,-247)(235,-237)
  \end{picture}
}}+ \fcolorbox{white}{white}{
\scalebox{0.7}{
  \begin{picture}(15,17) (216,-264)
    \SetWidth{0.0}
    \SetColor{Black}
    \Vertex(210,-260){2}
    \Vertex(226,-260){2}
    \Vertex(218,-249){2}
    \SetWidth{1.0}
    \Line(210,-261)(219,-248)
    \Line(226,-260)(218,-250)
    \Vertex(233,-260){2}
  \end{picture}
}}\otimes \fcolorbox{white}{white}{
\scalebox{0.7}{
  \begin{picture}(8,30) (236,-277)
    \SetWidth{1.0}
    \SetColor{Black}
    \Vertex(232,-258){2}
    \Vertex(239,-258){2}
    \Vertex(236,-253){2}
    \Arc(235,-256)(8.062,187,547)
    \Line(235,-264)(235,-275)
    \Vertex(235,-275){2}
  \end{picture}
}}+2\fcolorbox{white}{white}{
\scalebox{0.7}{
  \begin{picture}(3,17) (237,-291)
    \SetWidth{1.0}
    \SetColor{Black}
    \Line(235,-278)(235,-288)
    \SetWidth{0.0}
    \Vertex(235,-288){2}
    \SetWidth{1.0}
    \Vertex(235,-277){2}
    \Vertex(241,-288){2}
    \Line(241,-288)(241,-279)
    \Vertex(241,-277){2}
  \end{picture}
}}\otimes \fcolorbox{white}{white}{
\scalebox{0.7}{
  \begin{picture}(3,35) (236,-276)
    \SetWidth{1.0}
    \SetColor{Black}
    \Arc(233,-269)(6.403,141,501)
    \Arc(234,-249)(7,180,540)
    \Vertex(231,-270){2}
    \Vertex(236,-270){2}
    \Vertex(231,-250){2}
    \Vertex(236,-250){2}
    \Line(233,-256)(233,-263)
  \end{picture}
}}+\fcolorbox{white}{white}{
\scalebox{0.7}{
  \begin{picture}(10,7) (215,-286)
    \SetWidth{1.0}
    \SetColor{Black}
    \Vertex(208,-282){2}
    \Vertex(214,-282){2}
    \Vertex(220,-282){2}
    \Vertex(226,-282){2}
  \end{picture}
}}\otimes \fcolorbox{white}{white}{
\scalebox{0.7}{
  \begin{picture}(22,30) (218,-263)
    \SetWidth{1.0}
    \SetColor{Black}
    \Vertex(218,-259){2}
    \Vertex(210,-247){2}
    \Vertex(226,-247){2}
    \Vertex(218,-236){2}
    \Line(218,-260)(226,-246)
    \Line(219,-260)(209,-247)
    \Line(210,-248)(219,-235)
    \Line(226,-247)(218,-237)
  \end{picture}
}}\\
\hspace*{3cm} +\fcolorbox{white}{white}{
\scalebox{0.7}{
  \begin{picture}(8,30) (217,-263)
    \SetWidth{1.0}
    \SetColor{Black}
    \Vertex(218,-259){2}
    \Vertex(210,-247){2}
    \Vertex(226,-247){2}
    \Vertex(218,-236){2}
    \Line(218,-260)(226,-246)
    \Line(219,-260)(209,-247)
    \Line(210,-248)(219,-235)
    \Line(226,-247)(218,-237)
  \end{picture}
}} \otimes \fcolorbox{white}{white}{
\scalebox{0.7}{
  \begin{picture}(20,20) (206,-312)
    \SetWidth{1.0}
    \SetColor{Black}
    \Arc(208,-302)(9.22,167,527)
    \Vertex(205,-298){2}
    \Vertex(211,-298){2}
    \Vertex(205,-304){2}
    \Vertex(211,-304){2}
  \end{picture}
}}$\\
\\
then\\

$(Id\otimes \Delta_{\searrow})\Gamma(\fcolorbox{white}{white}{
\scalebox{0.7}{
  \begin{picture}(12,30) (216,-263)
    \SetWidth{1.0}
    \SetColor{Black}
    \Vertex(218,-259){2}
    \Vertex(210,-247){2}
    \Vertex(226,-247){2}
    \Vertex(218,-236){2}
    \Line(218,-260)(226,-246)
    \Line(219,-260)(209,-247)
    \Line(210,-248)(219,-235)
    \Line(226,-247)(218,-237)
  \end{picture}
}})=\fcolorbox{white}{white}{
\scalebox{0.7}{
  \begin{picture}(13,19) (217,-274)
    \SetWidth{1.0}
    \SetColor{Black}
    \Vertex(219,-270){2}
    \Vertex(210,-258){2}
    \Vertex(226,-258){2}
    \Line(218,-271)(226,-257)
    \Line(220,-270)(209,-258)
    \Vertex(232,-270){2}
  \end{picture}
}}\otimes [ \fcolorbox{white}{white}{
\scalebox{0.7}{
  \begin{picture}(-6,17) (215,-263)
    \SetWidth{1.0}
    \SetColor{Black}
    \Vertex(210,-259){2}
  \end{picture}
}} \otimes \fcolorbox{white}{white}{
\scalebox{0.7}{
  \begin{picture}(6,31) (236,-264)
    \SetWidth{1.0}
    \SetColor{Black}
    \Vertex(232,-257){2}
    \Vertex(239,-257){2}
    \Vertex(236,-252){2}
    \Arc(235,-255)(8.062,187,547)
  \end{picture}
}} + \mathbf{1}\otimes \fcolorbox{white}{white}{
\scalebox{0.7}{
  \begin{picture}(6,31) (236,-264)
    \SetWidth{1.0}
    \SetColor{Black}
    \Vertex(232,-257){2}
    \Vertex(239,-257){2}
    \Vertex(236,-252){2}
    \Arc(235,-255)(8.062,187,547)
    \Vertex(235,-237){2}
    \Line(235,-247)(235,-237)
  \end{picture}
}} + \fcolorbox{white}{white}{
\scalebox{0.7}{
  \begin{picture}(5,31) (236,-264)
    \SetWidth{1.0}
    \SetColor{Black}
    \Vertex(232,-257){2}
    \Vertex(239,-257){2}
    \Vertex(236,-252){2}
    \Arc(235,-255)(8.062,187,547)
    \Vertex(235,-237){2}
    \Line(235,-247)(235,-237)
  \end{picture}
}} \otimes \mathbf{1} ]\\
 \hspace*{5cm}+\fcolorbox{white}{white}{
 \scalebox{0.7}{
  \begin{picture}(15,17) (216,-264)
    \SetWidth{0.0}
    \SetColor{Black}
    \Vertex(210,-260){2}
    \Vertex(226,-260){2}
    \Vertex(218,-249){2}
    \SetWidth{1.0}
    \Line(210,-261)(219,-248)
    \Line(226,-260)(218,-250)
    \Vertex(233,-260){2}
  \end{picture}
}}\otimes [ \fcolorbox{white}{white}{
\scalebox{0.7}{
  \begin{picture}(6,31) (234,-264)
    \SetWidth{1.0}
    \SetColor{Black}
    \Vertex(232,-257){2}
    \Vertex(239,-257){2}
    \Vertex(236,-252){2}
    \Arc(235,-255)(8.062,187,547)
  \end{picture}
}} \otimes \fcolorbox{white}{white}{
\scalebox{0.7}{
  \begin{picture}(-9,17) (220,-263)
    \SetWidth{1.0}
    \SetColor{Black}
    \Vertex(210,-259){2}
  \end{picture}
}} + \mathbf{1}\otimes \fcolorbox{white}{white}{
\scalebox{0.7}{
  \begin{picture}(4,30) (238,-277)
    \SetWidth{1.0}
    \SetColor{Black}
    \Vertex(232,-258){2}
    \Vertex(239,-258){2}
    \Vertex(236,-253){2}
    \Arc(235,-256)(8.062,187,547)
    \Line(235,-264)(235,-275)
    \Vertex(235,-275){2}
  \end{picture}
}} + \fcolorbox{white}{white}{
\scalebox{0.7}{
  \begin{picture}(3,30) (236,-277)
    \SetWidth{1.0}
    \SetColor{Black}
    \Vertex(232,-258){2}
    \Vertex(239,-258){2}
    \Vertex(236,-253){2}
    \Arc(235,-256)(8.062,187,547)
    \Line(235,-264)(235,-275)
    \Vertex(235,-275){2}
  \end{picture}
}}\otimes \mathbf{1} ] \\
\hspace*{5cm}+2\fcolorbox{white}{white}{
\scalebox{0.7}{
  \begin{picture}(4,17) (236,-291)
    \SetWidth{1.0}
    \SetColor{Black}
    \Line(235,-278)(235,-288)
    \SetWidth{0.0}
    \Vertex(235,-288){2}
    \SetWidth{1.0}
    \Vertex(235,-277){2}
    \Vertex(241,-288){2}
    \Line(241,-288)(241,-279)
    \Vertex(241,-277){2}
  \end{picture}
}}\otimes [ \fcolorbox{white}{white}{
\scalebox{0.7}{
  \begin{picture}(1,35) (234,-276)
    \SetWidth{1.0}
    \SetColor{Black}
    \Arc(233,-269)(6.403,141,501)
    \Vertex(231,-270){2}
    \Vertex(236,-270){2}
  \end{picture}
}}\otimes \fcolorbox{white}{white}{
\scalebox{0.7}{
  \begin{picture}(3,35) (236,-276)
    \SetWidth{1.0}
    \SetColor{Black}
    \Arc(233,-269)(6.403,141,501)
    \Vertex(231,-270){2}
    \Vertex(236,-270){2}
  \end{picture}
}} + \mathbf{1}\otimes \fcolorbox{white}{white}{
\scalebox{0.7}{
  \begin{picture}(3,35) (236,-276)
    \SetWidth{1.0}
    \SetColor{Black}
    \Arc(233,-269)(6.403,141,501)
    \Arc(234,-249)(7,180,540)
    \Vertex(231,-270){2}
    \Vertex(236,-270){2}
    \Vertex(231,-250){2}
    \Vertex(236,-250){2}
    \Line(233,-256)(233,-263)
  \end{picture}
}}+ \fcolorbox{white}{white}{
\scalebox{0.7}{
  \begin{picture}(3,35) (234,-276)
    \SetWidth{1.0}
    \SetColor{Black}
    \Arc(233,-269)(6.403,141,501)
    \Arc(234,-249)(7,180,540)
    \Vertex(231,-270){2}
    \Vertex(236,-270){2}
    \Vertex(231,-250){2}
    \Vertex(236,-250){2}
    \Line(233,-256)(233,-263)
  \end{picture}
}}\otimes \mathbf{1} ]\\
\hspace*{5cm}+\fcolorbox{white}{white}{
\scalebox{0.7}{
  \begin{picture}(11,7) (213,-286)
    \SetWidth{1.0}
    \SetColor{Black}
    \Vertex(208,-282){2}
    \Vertex(214,-282){2}
    \Vertex(220,-282){2}
    \Vertex(226,-282){2}
  \end{picture}
}}\otimes [ \fcolorbox{white}{white}{
\scalebox{0.7}{
  \begin{picture}(8,17) (216,-263)
    \SetWidth{1.0}
    \SetColor{Black}
    \Vertex(210,-260){2}
    \Vertex(226,-260){2}
    \Vertex(218,-249){2}
    \Line(210,-261)(219,-248)
    \Line(226,-260)(218,-250)
  \end{picture}
}} \otimes \fcolorbox{white}{white}{
\scalebox{0.7}{
  \begin{picture}(-6,17) (217,-263)
    \SetWidth{1.0}
    \SetColor{Black}
    \Vertex(210,-259){2}
  \end{picture}
}} + \mathbf{1}\otimes  \fcolorbox{white}{white}{
\scalebox{0.7}{
  \begin{picture}(10,30) (219,-263)
    \SetWidth{1.0}
    \SetColor{Black}
    \Vertex(218,-259){2}
    \Vertex(210,-247){2}
    \Vertex(226,-247){2}
    \Vertex(218,-236){2}
    \Line(218,-260)(226,-246)
    \Line(219,-260)(209,-247)
    \Line(210,-248)(219,-235)
    \Line(226,-247)(218,-237)
  \end{picture}
}} + \fcolorbox{white}{white}{
\scalebox{0.7}{
  \begin{picture}(8,30) (217,-263)
    \SetWidth{1.0}
    \SetColor{Black}
    \Vertex(218,-259){2}
    \Vertex(210,-247){2}
    \Vertex(226,-247){2}
    \Vertex(218,-236){2}
    \Line(218,-260)(226,-246)
    \Line(219,-260)(209,-247)
    \Line(210,-248)(219,-235)
    \Line(226,-247)(218,-237)
  \end{picture}
}}  \otimes \mathbf{1} ]\\
\hspace*{5cm}+\fcolorbox{white}{white}{
\scalebox{0.7}{
  \begin{picture}(11,30) (215,-263)
    \SetWidth{1.0}
    \SetColor{Black}
    \Vertex(218,-259){2}
    \Vertex(210,-247){2}
    \Vertex(226,-247){2}
    \Vertex(218,-236){2}
    \Line(218,-260)(226,-246)
    \Line(219,-260)(209,-247)
    \Line(210,-248)(219,-235)
    \Line(226,-247)(218,-237)
  \end{picture}
}} \otimes [ \fcolorbox{white}{white}{
\scalebox{0.7}{
  \begin{picture}(10,20) (205,-312)
    \SetWidth{1.0}
    \SetColor{Black}
    \Arc(208,-302)(9.22,167,527)
    \Vertex(205,-298){2}
    \Vertex(211,-298){2}
    \Vertex(205,-304){2}
    \Vertex(211,-304){2}
  \end{picture}
}}\otimes \mathbf{1} + \mathbf{1}\otimes \fcolorbox{white}{white}{
\scalebox{0.7}{
  \begin{picture}(9,20) (209,-312)
    \SetWidth{1.0}
    \SetColor{Black}
    \Arc(208,-302)(9.22,167,527)
    \Vertex(205,-298){2}
    \Vertex(211,-298){2}
    \Vertex(205,-304){2}
    \Vertex(211,-304){2}
  \end{picture}
}} ]$\\
\\
On the other hand\\

$m^{13}(\Gamma \otimes \Gamma)\Delta_{\searrow}(\fcolorbox{white}{white}{
\scalebox{0.7}{
  \begin{picture}(11,30) (215,-263)
    \SetWidth{1.0}
    \SetColor{Black}
    \Vertex(218,-259){2}
    \Vertex(210,-247){2}
    \Vertex(226,-247){2}
    \Vertex(218,-236){2}
    \Line(218,-260)(226,-246)
    \Line(219,-260)(209,-247)
    \Line(210,-248)(219,-235)
    \Line(226,-247)(218,-237)
  \end{picture}
}})=\fcolorbox{white}{white}{
\scalebox{0.7}{
  \begin{picture}(13,19) (217,-274)
    \SetWidth{1.0}
    \SetColor{Black}
    \Vertex(219,-270){2}
    \Vertex(210,-258){2}
    \Vertex(226,-258){2}
    \Line(218,-271)(226,-257)
    \Line(220,-270)(209,-258)
    \Vertex(232,-270){2}
  \end{picture}
}}\otimes [ \mathbf{1}\otimes \fcolorbox{white}{white}{
\scalebox{0.7}{
  \begin{picture}(8,31) (236,-264)
    \SetWidth{1.0}
    \SetColor{Black}
    \Vertex(232,-257){2}
    \Vertex(239,-257){2}
    \Vertex(236,-252){2}
    \Arc(235,-255)(8.062,187,547)
    \Vertex(235,-237){2}
    \Line(235,-247)(235,-237)
  \end{picture}
}} +\fcolorbox{white}{white}{
\scalebox{0.7}{
  \begin{picture}(4,31) (236,-264)
    \SetWidth{1.0}
    \SetColor{Black}
    \Vertex(232,-257){2}
    \Vertex(239,-257){2}
    \Vertex(236,-252){2}
    \Arc(235,-255)(8.062,187,547)
    \Vertex(235,-237){2}
    \Line(235,-247)(235,-237)
  \end{picture}
}} \otimes \mathbf{1} ]\\
 \hspace*{5cm}+\fcolorbox{white}{white}{
 \scalebox{0.7}{
  \begin{picture}(15,17) (216,-264)
    \SetWidth{0.0}
    \SetColor{Black}
    \Vertex(210,-260){2}
    \Vertex(226,-260){2}
    \Vertex(218,-249){2}
    \SetWidth{1.0}
    \Line(210,-261)(219,-248)
    \Line(226,-260)(218,-250)
    \Vertex(233,-260){2}
  \end{picture}
}}\otimes [ \fcolorbox{white}{white}{
\scalebox{0.7}{
  \begin{picture}(5,31) (235,-264)
    \SetWidth{1.0}
    \SetColor{Black}
    \Vertex(232,-257){2}
    \Vertex(239,-257){2}
    \Vertex(236,-252){2}
    \Arc(235,-255)(8.062,187,547)
  \end{picture}
}} \otimes  \fcolorbox{white}{white}{
\scalebox{0.7}{
  \begin{picture}(-6,17) (217,-263)
    \SetWidth{1.0}
    \SetColor{Black}
    \Vertex(210,-259){2}
  \end{picture}
}}  + \mathbf{1}  \otimes \fcolorbox{white}{white}{
\scalebox{0.7}{
  \begin{picture}(6,30) (236,-277)
    \SetWidth{1.0}
    \SetColor{Black}
    \Vertex(232,-258){2}
    \Vertex(239,-258){2}
    \Vertex(236,-253){2}
    \Arc(235,-256)(8.062,187,547)
    \Line(235,-264)(235,-275)
    \Vertex(235,-275){2}
  \end{picture}
}}  +  \fcolorbox{white}{white}{
\scalebox{0.7}{
  \begin{picture}(4,30) (236,-277)
    \SetWidth{1.0}
    \SetColor{Black}
    \Vertex(232,-258){2}
    \Vertex(239,-258){2}
    \Vertex(236,-253){2}
    \Arc(235,-256)(8.062,187,547)
    \Line(235,-264)(235,-275)
    \Vertex(235,-275){2}
  \end{picture}
}}\otimes \mathbf{1} ] \\
\hspace*{5cm}+2\fcolorbox{white}{white}{
\scalebox{0.7}{
  \begin{picture}(3,17) (237,-291)
    \SetWidth{1.0}
    \SetColor{Black}
    \Line(235,-278)(235,-288)
    \SetWidth{0.0}
    \Vertex(235,-288){2}
    \SetWidth{1.0}
    \Vertex(235,-277){2}
    \Vertex(241,-288){2}
    \Line(241,-288)(241,-279)
    \Vertex(241,-277){2}
  \end{picture}
}}\otimes [ \mathbf{1}\otimes \fcolorbox{white}{white}{
\scalebox{0.7}{
  \begin{picture}(2,35) (236,-276)
    \SetWidth{1.0}
    \SetColor{Black}
    \Arc(233,-269)(6.403,141,501)
    \Arc(234,-249)(7,180,540)
    \Vertex(231,-270){2}
    \Vertex(236,-270){2}
    \Vertex(231,-250){2}
    \Vertex(236,-250){2}
    \Line(233,-256)(233,-263)
  \end{picture}
}}+ \fcolorbox{white}{white}{
\scalebox{0.7}{
  \begin{picture}(3,35) (234,-276)
    \SetWidth{1.0}
    \SetColor{Black}
    \Arc(233,-269)(6.403,141,501)
    \Arc(234,-249)(7,180,540)
    \Vertex(231,-270){2}
    \Vertex(236,-270){2}
    \Vertex(231,-250){2}
    \Vertex(236,-250){2}
    \Line(233,-256)(233,-263)
  \end{picture}
}}\otimes \mathbf{1} ]\\
\hspace*{5cm}+\fcolorbox{white}{white}{
\scalebox{0.7}{
  \begin{picture}(9,7) (214,-286)
    \SetWidth{1.0}
    \SetColor{Black}
    \Vertex(208,-282){2}
    \Vertex(214,-282){2}
    \Vertex(220,-282){2}
    \Vertex(226,-282){2}
  \end{picture}
}}\otimes [ \fcolorbox{white}{white}{
\scalebox{0.7}{
  \begin{picture}(9,17) (215,-263)
    \SetWidth{1.0}
    \SetColor{Black}
    \Vertex(210,-260){2.236}
    \Vertex(226,-260){2}
    \Vertex(218,-249){2}
    \Line(210,-261)(219,-248)
    \Line(226,-260)(218,-250)
  \end{picture}
}} \otimes  \fcolorbox{white}{white}{
\scalebox{0.7}{
  \begin{picture}(-5,17) (217,-263)
    \SetWidth{1.0}
    \SetColor{Black}
    \Vertex(210,-259){2}
  \end{picture}
}} + \mathbf{1}\otimes  \fcolorbox{white}{white}{
\scalebox{0.7}{
  \begin{picture}(8,30) (219,-263)
    \SetWidth{1.0}
    \SetColor{Black}
    \Vertex(218,-259){2}
    \Vertex(210,-247){2}
    \Vertex(226,-247){2}
    \Vertex(218,-236){2}
    \Line(218,-260)(226,-246)
    \Line(219,-260)(209,-247)
    \Line(210,-248)(219,-235)
    \Line(226,-247)(218,-237)
  \end{picture}
}} +  \fcolorbox{white}{white}{
\scalebox{0.7}{
  \begin{picture}(7,30) (217,-263)
    \SetWidth{1.0}
    \SetColor{Black}
    \Vertex(218,-259){2}
    \Vertex(210,-247){2}
    \Vertex(226,-247){2}
    \Vertex(218,-236){2}
    \Line(218,-260)(226,-246)
    \Line(219,-260)(209,-247)
    \Line(210,-248)(219,-235)
    \Line(226,-247)(218,-237)
  \end{picture}
}}  \otimes \mathbf{1} ]\\
\hspace*{5cm}+\fcolorbox{white}{white}{
\scalebox{0.7}{
  \begin{picture}(8,30) (217,-263)
    \SetWidth{1.0}
    \SetColor{Black}
    \Vertex(218,-259){2}
    \Vertex(210,-247){2}
    \Vertex(226,-247){2}
    \Vertex(218,-236){2}
    \Line(218,-260)(226,-246)
    \Line(219,-260)(209,-247)
    \Line(210,-248)(219,-235)
    \Line(226,-247)(218,-237)
  \end{picture}
}} \otimes [ \fcolorbox{white}{white}{
\scalebox{0.7}{
  \begin{picture}(9,20) (205,-312)
    \SetWidth{1.0}
    \SetColor{Black}
    \Arc(208,-302)(9.22,167,527)
    \Vertex(205,-298){2}
    \Vertex(211,-298){2}
    \Vertex(205,-304){2}
    \Vertex(211,-304){2}
  \end{picture}
}}\otimes \mathbf{1} + \mathbf{1}\otimes \fcolorbox{white}{white}{
\scalebox{0.7}{
  \begin{picture}(11,20) (205,-312)
    \SetWidth{1.0}
    \SetColor{Black}
    \Arc(208,-302)(9.22,167,527)
    \Vertex(205,-298){2}
    \Vertex(211,-298){2}
    \Vertex(205,-304){2}
    \Vertex(211,-304){2}
  \end{picture}
}} ]\\
\hspace*{5cm}+\fcolorbox{white}{white}{
\scalebox{0.7}{
  \begin{picture}(5,19) (232,-262)
    \SetWidth{0.0}
    \SetColor{Black}
    \Vertex(226,-258){2}
    \Vertex(233,-258){2}
    \SetWidth{1.0}
    \Line(227,-245)(228,-244)
    \SetWidth{0.0}
    \Vertex(226,-245){2}
    \SetWidth{1.0}
    \Line(226,-259)(226,-245)
    \Vertex(239,-258){2}
  \end{picture}
}}\otimes [\fcolorbox{white}{white}{
\scalebox{0.7}{
  \begin{picture}(2,24) (234,-287)
    \SetWidth{1.0}
    \SetColor{Black}
    \Arc(233,-280)(6.403,141,501)
    \SetWidth{0.0}
    \Vertex(231,-281){2}
    \Vertex(236,-281){2}
    \SetWidth{1.0}
    \Line(233,-267)(233,-274)
    \Vertex(233,-266){2}
  \end{picture}
}}\otimes  \fcolorbox{white}{white}{
\scalebox{0.7}{
  \begin{picture}(-6,17) (217,-263)
    \SetWidth{1.0}
    \SetColor{Black}
    \Vertex(210,-259){2}
  \end{picture}
}} + \fcolorbox{white}{white}{
\scalebox{0.7}{
  \begin{picture}(2,26) (236,-276)
    \SetWidth{1.0}
    \SetColor{Black}
    \Arc(234,-258)(7,180,540)
    \SetWidth{0.0}
    \Vertex(231,-259){2}
    \Vertex(236,-259){2}
    \SetWidth{1.0}
    \Line(234,-265)(234,-272)
    \Vertex(234,-273){2}
  \end{picture}
}} \otimes  \fcolorbox{white}{white}{
\scalebox{0.7}{
  \begin{picture}(12,17) (217,-263)
    \SetWidth{1.0}
    \SetColor{Black}
    \Vertex(210,-259){2}
  \end{picture}
}}\hspace*{-0,5cm}] $\\
then $(Id\otimes \Delta_{\searrow})\Gamma(\fcolorbox{white}{white}{
\scalebox{0.7}{
  \begin{picture}(11,30) (216,-263)
    \SetWidth{1.0}
    \SetColor{Black}
    \Vertex(218,-259){2}
    \Vertex(210,-247){2}
    \Vertex(226,-247){2}
    \Vertex(218,-236){2}
    \Line(218,-260)(226,-246)
    \Line(219,-260)(209,-246)
    \Line(210,-248)(219,-235)
    \Line(226,-247)(218,-237)
  \end{picture}
}})\neq m^{13}(\Gamma \otimes \Gamma)\Delta_{\searrow}(\fcolorbox{white}{white}{
\scalebox{0.7}{
  \begin{picture}(11,30) (216,-263)
    \SetWidth{1.0}
    \SetColor{Black}
    \Vertex(218,-259){2}
    \Vertex(210,-247){2}
    \Vertex(226,-247){2}
    \Vertex(218,-236){2}
    \Line(218,-260)(226,-246)
    \Line(219,-260)(209,-247)
    \Line(210,-248)(219,-235)
    \Line(226,-247)(218,-237)
  \end{picture}
}})$.
Then $\Gamma \hbox{ and } \Delta_{\searrow}$ are not compatible.
\end{remark}
\subsection{Relation between $\star$ and  $\Delta_{\searrow}$}
In this subsection, we prove that there exist relations between the Grossman-Larson product $\star$ and the  coproduct $\Delta_{\searrow}$.\\
Let $G$ be a group acting on $X$. For every $x\in X$, we denote by $G\cdot x$ the orbit of $x$ and we denote by $G_x$ the stabilizer subgroup of $G$ with respect to $x$. The group action is transitive if and only if it has exactly one orbit, that is if there exists $x$ in $X$ with $G\cdot x=X$ (i.e. $X$ is non-empty and if for each pair $x, y \in X$ there exists $g\in G$ such that $g\cdot x=y$). This is the case if and only if $G\cdot x=X$, for all $x$ in $X$.\\
     If $G$ and $X$ is finite, then the orbit-stabilizer theorem, together with Lagrange's theorem \cite{acg..0}(theorem 3.9), gives
    \begin{equation}
        {\displaystyle |G\cdot x|=[G_{X}\,:\,G]=\dfrac{|G|}{|G_{x}|}},
    \end{equation}
 that implies that the cardinal of the orbit is a divisor of the group order.

\begin{definition}
    For any topology $\mathcal{T}$ on a finite set $X$,
    we denote by $\hbox{\hbox{Aut}}(\mathcal{T})$ the subgroup of permutations of $X$ which are homeomorphisms with respect to $\mathcal{T}$. The symmetry factor is defined by $\sigma(\mathcal{T})=| \hbox{\hbox{Aut}}(\mathcal{T}) |$.
    We define the linear map $e_{\mathcal{T}}:\mathbb{T}_{X} \longrightarrow \mathbb{K}$ by:
    $$e_{\mathcal{T}}(\mathcal{T}')=\sigma(\mathcal{T}), \hbox{ if } \mathcal{T}=\mathcal{T}', \hbox{ and 0 if not}.$$
\end{definition}
\begin{definition}
    We define the graft operator
    $\mathcal{B}:\mathbb{T}\to \mathbb{T}$ by, $\mathcal{B}(\mathcal{T})=\mathcal{T}\searrow \{*\}$, for any topology $\mathcal{T} \hbox{ on } X$, this is the topology on $X\sqcup \{*\}$ obtained by keeping the preorder and $X$ and by putting $*< x$ for any $x\in X$.
\end{definition}
\begin{theorem}
Let $\mathcal{T}_1\in \mathbb{T}_{X_1}$, $\mathcal{T}_2\in \mathbb{T}_{X_2}$ and $\mathcal{T}'\in \mathbb{T}_{X}$, then
$$<e_{\mathcal{T}_1}\star e_{\mathcal{T}_2}, \mathcal{T}'>=<e_{\mathcal{T}_1}\otimes e_{\mathcal{T}_2}, \Delta_{\searrow}(\mathcal{T}')>.$$
\end{theorem}
\begin{proof}
 Let $\mathcal{T}_1\in \mathbb{T}_{X_1}$, $\mathcal{T}_2\in \mathbb{T}_{X_2}$ and $\mathcal{T}'\in \mathbb{T}_{X}$.\\
 \textbf{Case 1}; $\mathcal{T}_1$ is connected, we have
 \begin{align*}
  <e_{\mathcal{T}_1}\star e_{\mathcal{T}_2}, \mathcal{T}'>&=<e_{\mathcal{T}_1 \searrow \mathcal{B}\mathcal{T}_2}, \mathcal{B}\mathcal{T}'>\\
  &=\sum_{v\in X_2\sqcup \{*\}} <e_{\mathcal{T}_1 \searrow_v \mathcal{B}\mathcal{T}_2}, \mathcal{B}\mathcal{T}'>\\
  &=\sum \limits_{\underset{\scriptstyle\mathcal{B}\mathcal{T}'=\mathcal{T}_1 \searrow_v \mathcal{B}\mathcal{T}_2}{v\in X_2\sqcup \{*\} }} \sigma(\mathcal{B}\mathcal{T}').
  \end{align*} 
  Let us consider the set $B=\{v\in X_2\sqcup \{*\}, \mathcal{B}\mathcal{T}'=\mathcal{T}_1 \searrow_v \mathcal{B}\mathcal{T}_2 \}$,
  we show that $\hbox{\hbox{Aut}}(\mathcal{B}\mathcal{T}_2)$ acts transitively on $B$.
  We define the map 
  \begin{eqnarray*}
		\Phi_1: \hbox{\hbox{Aut}}(\mathcal{B}\mathcal{T}_2)\times B&\longrightarrow& B\\
		(\varphi, v)&\longmapsto& \varphi(v).
	\end{eqnarray*}
	Let $v\in B$, if $v=*$, then $\varphi(v)=v$. If not then $\mathcal{T}_1 \searrow_{\varphi(v)} \mathcal{B}\mathcal{T}_2=\mathcal{T}_1 \searrow_{\varphi(v)} \varphi(\mathcal{B}\mathcal{T}_2)=h(\mathcal{T}_1 \searrow_v \mathcal{B}\mathcal{T}_2)$, where $h_{|X_1}=Id$ and if $v\in X_2 \sqcup \{*\}$, $h(v)=\varphi(v)$, where $\varphi \in \hbox{Aut}(\mathcal{B}\mathcal{T}_2)$. It is clear that $h\in \hbox{Aut}(\mathcal{B}\mathcal{T}')$, 
	then
	$\mathcal{T}_1 \searrow_{\varphi(v)} \mathcal{B}\mathcal{T}_2=h(\mathcal{B}\mathcal{T}')=\mathcal{B}\mathcal{T}'$.
	Then $\Phi_1$ is well defined.\\
	Moreover for all $v\in B$, $Id(v)=v$, and for all $\varphi$, $\varphi'\in \hbox{\hbox{Aut}}(\mathcal{B}\mathcal{T}_2)$,
	$\Phi_1 \big(\varphi, \varphi'(v)  \big)=\varphi(\varphi'(v))=(\varphi \varphi')(v)$. Then $\Phi_1$ is an action.\\
	Now to show that $\Phi_1$ is transitive,
	let $u, v \in B$, and let us define $f:X\sqcup \{*\}\longrightarrow X\sqcup \{*\}$ by\\
	$f(u)=v$, $f(v)=u$, and for all $w\in X\sqcup \{*\} \backslash \{u, v\}, f(w)=w$, it is clear that $f\in \hbox{\hbox{Aut}}(\mathcal{B}\mathcal{T}')$.\\
	If we take $\varphi: X_2\sqcup \{*\}\longrightarrow X_2\sqcup \{*\}$, defined by $\varphi=f_{|X_2\sqcup \{*\}}$, so we have, $\varphi \in \hbox{\hbox{Aut}}(\mathcal{B}\mathcal{T}_2)$, and $\varphi(u)=v$, then $\Phi_1$ is transitive. Then $B=\hbox{\hbox{Aut}}(\mathcal{B}\mathcal{T}_2)\cdot v, \hbox{ for all } v\in B$.\\
	For all $v\in B$, we call the stabilizer of $v$ the set: 
	$$ \hbox{\hbox{Aut}}(\mathcal{B}\mathcal{T}_2)_v=\{ \varphi \in \hbox{\hbox{Aut}}(\mathcal{B}\mathcal{T}_2), \varphi(v)=v\}.$$
	And since $|(\hbox{\hbox{Aut}}(\mathcal{B}\mathcal{T}_2)|$ is finite, then $|B|=|\hbox{\hbox{Aut}}(\mathcal{B}\mathcal{T}_2)\cdot v|=\dfrac{|\hbox{\hbox{Aut}}(\mathcal{B}\mathcal{T}_2)|}{| \hbox{\hbox{Aut}}(\mathcal{B}\mathcal{T}_2)_v |}$, for all $v\in B$.\\
	Then $$|B|=\dfrac{|\sigma(\mathcal{T}_2)|}{| \hbox{\hbox{Aut}}(\mathcal{T}_2)_v |}, \hbox{ for all } v\in B\backslash \{*\}.$$
	 On the other hand
  \begin{align*}
    <e_{\mathcal{T}_1}\otimes e_{\mathcal{T}_2}, \Delta_{\searrow}(\mathcal{T}')>&=\sum_{Y\overline{\in}\mathcal{T}'}<e_{\mathcal{T}_1}, \mathcal{T}^{\prime}_{|Y}><e_{\mathcal{T}_2}, \mathcal{T}^{\prime}_{|X\backslash Y}>\\
    &=\sum \limits_{\underset{\scriptstyle\mathcal{T}_1=\mathcal{T}_{|Y}^{\prime},\, \mathcal{T}_2=\mathcal{T}^{\prime}_{|X\backslash Y}}{Y\overline{\in}\mathcal{T}'  }} \sigma({\mathcal{T}_1})\sigma({\mathcal{T}_2}).
  \end{align*}
  Let us consider the set
  $$A=\{ v\in X, \hbox{ the cut above } v \hbox{ give the term of } \Delta_{\searrow}(\mathcal{T}') \hbox{ isomorphic to }
  \mathcal{T}_1\otimes \mathcal{T}_2 \},$$
   we notice that $A\cap B\backslash \{*\} \neq \varnothing$. \\
    We show that $\hbox{Aut}(\mathcal{T}^{\prime})$ acts transitively on $A$.
  We define the map 
  \begin{eqnarray*}
		\Phi_2: \hbox{Aut}(\mathcal{T}')\times A&\longrightarrow& A\\
		(\varphi, v)&\longmapsto& \varphi(v).
	\end{eqnarray*}
	  Let $v\in A$ then 
  $\mathcal{T}_{|X_1}^{\prime} \hbox{ isomorphic to }\mathcal{T}_1$ and $\mathcal{T}_{|X_2}^{\prime}\hbox{ isomorphic to }\mathcal{T}_2$, then for all $\varphi \in \hbox{Aut}(\mathcal{T}')$, $\varphi(\mathcal{T}_{|X_i}^{\prime})$ isomorph to 
  $\mathcal{T}_i,
  i\in \{ 1, 2 \}$, and  $\Delta_{\searrow}\big( \varphi(\mathcal{T}')\big) \hbox{ isomorphic to } \Delta_{\searrow}( \mathcal{T}')$,\\
  then for all $\varphi \in \hbox{\hbox{Aut}}(\mathcal{T}'), \hbox{ the cut above } \varphi(v) \hbox{ give the term of } \Delta_{\searrow}(\mathcal{T}') \hbox{ isomorphic to } \mathcal{T}_1\otimes \mathcal{T}_2$, then\\
  $\varphi(v)\in A$. Then $\Phi_2$ is well defined.\\
	Let $v\in A$, $\hbox{ and } \varphi, \varphi' \in \hbox{Aut}(\mathcal{T}')$, then $Id(v)=v$ and $\Phi_2 \big( \varphi, \varphi'(v)  \big)=\varphi(\varphi'(v))=(\varphi \varphi')(v)$. Then $\Phi_2$ is an action.\\
	Let $u, v \in A$, we defined $f:X\longrightarrow X$ by:
	$f(u)=v$, $f(v)=u$, and for all $w\notin \{u, v\},
	f(w)=w$, it is clear that $f\in \hbox{Aut}(\mathcal{T}')$, then $\Phi_2$ is transitive.
		And since $|\hbox{Aut}(\mathcal{T}')|$ is finite, then
		$$|A|=\dfrac{|\hbox{Aut}(\mathcal{T}')|}{| \hbox{Aut}(\mathcal{T}')_v |}, \hbox{ for all } v\in A.$$
	Let $v\in A$, then $| \hbox{Aut}(\mathcal{T}')_v |=| \hbox{Aut}(\mathcal{T}_1)\hbox{Aut}(\mathcal{T}_2)_v |=| \hbox{Aut}(\mathcal{T}_1)| |\hbox{Aut}(\mathcal{T}_2)_v |$. Then\\
	$$|\hbox{Aut}(\mathcal{T}_2)_v |= \dfrac{|\hbox{Aut}(\mathcal{T}')|}{|A|| \hbox{Aut}(\mathcal{T}_1)|}, \hbox{ for all } v\in A,$$
	that since $A\cap B\backslash \{*\} \neq \varnothing$, there exists $v\in A\cap B\backslash \{*\}$ such that 
	$$|\hbox{Aut}(\mathcal{T}_2)_v |= \dfrac{|\hbox{Aut}(\mathcal{T}')|}{|A|| \hbox{Aut}(\mathcal{T}_1)|}=\dfrac{|\hbox{Aut}(\mathcal{T}_2)|}{|B|}$$
	then 
	$$\dfrac{\sigma(\mathcal{T}')}{|A| \sigma(\mathcal{T}_1)}=\dfrac{\sigma(\mathcal{T}_2)}{|B|},$$
	then
	$$|B|\sigma(\mathcal{T}')=|A| \sigma(\mathcal{T}_1)\sigma(\mathcal{T}_2).$$
	We define $A'=\{ Y, Y\overline{\in}\mathcal{T}', \mathcal{T}_1=\mathcal{T}^{\prime}_Y, \mathcal{T}_2=\mathcal{T}^{\prime}_{|X\backslash Y} \}$,
	we notice that $|A|=|A'|$.
	Then
	\begin{align*}
  <e_{\mathcal{T}_1}\star e_{\mathcal{T}_2}, \mathcal{T}'>&=\sum \limits_{\underset{\scriptstyle\mathcal{B}\mathcal{T}'=\mathcal{T}_1 \searrow_v \mathcal{B}\mathcal{T}_2}{v\in X_2\sqcup \{*\} }} \sigma(\mathcal{B}\mathcal{T}')\\
  &=|B|\sigma(\mathcal{B}\mathcal{T}')\\
  &=|B|\sigma(\mathcal{T}')\\
  &=|A'| \sigma(\mathcal{T}_1)\sigma(\mathcal{T}_2)\\
  &=\sum \limits_{\underset{\scriptstyle\mathcal{T}_1=\mathcal{T}^{\prime}_{|Y},\, \mathcal{T}_2=\mathcal{T}^{\prime}_{|X\backslash Y}}{Y\overline{\in}\mathcal{T}'  }} \sigma(\mathcal{T}_1)\sigma(\mathcal{T}_2)\\
  &=<e_{\mathcal{T}_1}\otimes e_{\mathcal{T}_2}, \Delta_{\searrow}(\mathcal{T}')>.
  \end{align*} 
  \\
  \textbf{Case2}; $\mathcal{T}_1$ not connected.\\
  Let $\mathcal{T}_1=\mathcal{T}_{1,1}...\mathcal{T}_{1,n}\in \mathbb{T}_{X_1}$, where $\mathcal{T}_{1,i}$ is connected for all $i\in [n]$. And let  $\mathcal{T}_2\in \mathbb{T}_{X_2}$, $\mathcal{T}'\in \mathbb{T}_{X}$, we have
\begin{align*}
  <e_{\mathcal{T}_1}\star e_{\mathcal{T}_2}, \mathcal{T}'>&=<e_{\mathcal{T}_1 \searrow \mathcal{B}\mathcal{T}_2}, \mathcal{B}\mathcal{T}'>\\
  &=\sum_{\underline{v}=(v_1,..., v_n)\in X_2} <e_{\mathcal{T}_1 \searrow_{\underline{v}} \mathcal{B}\mathcal{T}_2}, \mathcal{B}\mathcal{T}'>\\
  &=\sum \limits_{\underset{\scriptstyle\mathcal{B}\mathcal{T}'=\mathcal{T}_{1,1} \searrow_{v_1}(\mathcal{T}_{1,2} \searrow_{v_2}...(\mathcal{T}_{1,n} \searrow_{v_n} \mathcal{B}\mathcal{T}_2)...)}{\underline{v}=(v_1,..., v_n)\in X_2}} \sigma(\mathcal{B}\mathcal{T}').
  \end{align*}
   Let us consider the set $\underline{B}=\{ \underline{v}=(v_1,..., v_n)\in X_2\sqcup \{*\}, \mathcal{B}\mathcal{T}'=\mathcal{T}_1 \searrow_{\underline{v}} \mathcal{B}\mathcal{T}_2 \}$.\\
 $\hbox{Aut}(\mathcal{T}_2)$ acts transitively on $\underline{B}$ by the action. (In the same way that we used to show that $\Phi_1$ is a transitive action.)
 \begin{eqnarray*}
		\Phi_3: \hbox{Aut}(\mathcal{T}_2)\times \underline{B}&\longrightarrow& \underline{B}\\
		(\varphi, \underline{v})&\longmapsto& \varphi(\underline{v}).
	\end{eqnarray*}
	And since $|\hbox{Aut}(\mathcal{T}_2)|$ is finite, then 
	$|\underline{B}|=\dfrac{\sigma(\mathcal{T}_2)}{| \hbox{Aut}(\mathcal{T}_2)_{\underline{v}} |}$, where $\underline{v}\in \underline{B}$.\\
	On the other hand
  \begin{align*}
    <e_{\mathcal{T}_1}\otimes e_{\mathcal{T}_2}, \Delta_{\searrow}(\mathcal{T}')>&=\sum_{Y\overline{\in}\mathcal{T}'}<e_{\mathcal{T}_1}, \mathcal{T}^{\prime}_{|Y}><e_{\mathcal{T}_2}, \mathcal{T}^{\prime}_{|X\backslash Y}>\\
    &=\sum \limits_{\underset{\scriptstyle\mathcal{T}_1=\mathcal{T}^{\prime}_{|Y},\, \mathcal{T}_2=\mathcal{T}^{\prime}_{|X\backslash Y}}{Y\overline{\in}\mathcal{T}'  }} \sigma({\mathcal{T}_1})\sigma({\mathcal{T}_2}).
  \end{align*}
  Let us consider the set\\
  $\underline{A}=\{ \underline{v}=(v_1,...,v_n)\in X, \hbox{ the cut above } \underline{v} \hbox{ give the term of } \Delta_{\searrow}(\mathcal{T}') \hbox{ isomorphic } \hbox{ to }\mathcal{T}_1\otimes \mathcal{T}_2 \}$.\\
   We notice that $\underline{A}\cap \underline{B}_{|X_2} \neq \varnothing$.\\
$\hbox{Aut}(\mathcal{T}')$ acts transitively on $\underline{A}$ by the action. (In the same way that we used to show that $\Phi_2$ is a transitive action.)
 \begin{eqnarray*}
\Phi_4: \hbox{Aut}(\mathcal{T}')\times \underline{A}&\longrightarrow& \underline{A}\\
(\varphi, \underline{v})&\longmapsto& \varphi(\underline{v}).
\end{eqnarray*}	
		And since $|\hbox{Aut}(\mathcal{T}')|$ is finite, then 
	$|\underline{A}|=\dfrac{\sigma(\mathcal{T}')}{| \hbox{Aut}(\mathcal{T}')_{\underline{v}} |}$, where $\underline{v}\in \underline{A}$.\\
	If $\underline{v}\in \underline{A}$, then $| \hbox{Aut}(\mathcal{T}')_{\underline{v}} |=| \hbox{Aut}(\mathcal{T}_1)\hbox{Aut}(\mathcal{T}_2)_{\underline{v}} |=| \hbox{Aut}(\mathcal{T}_1)| | \hbox{Aut}(\mathcal{T}_2)_{\underline{v}} |$,\\
	then $|\hbox{Aut}(\mathcal{T}_2)_{\underline{v}} |= \dfrac{|\hbox{Aut}(\mathcal{T}')|}{|\underline{A}|| \hbox{Aut}(\mathcal{T}_1)|}$ for all $v\in \underline{A}$,\\
	then 
	$$|\hbox{Aut}(\mathcal{T}_2)_{\underline{v}} |= \dfrac{|\hbox{Aut}(\mathcal{T}')|}{|\underline{A}|| \hbox{Aut}(\mathcal{T}_1)|}=\dfrac{|\hbox{Aut}(\mathcal{T}_2)|}{|\underline{B}|}, \hbox{ for all }v\in \underline{A}\cap \underline{B}\cap X_2,$$ 
	then  $\dfrac{\sigma(\mathcal{T}')}{|\underline{A}| \sigma(\mathcal{T}_1)}=\dfrac{\sigma(\mathcal{T}_2)}{|\underline{B}|}$. 
	Then $|\underline{B}|\sigma(\mathcal{T}')=|\underline{A}| \sigma(\mathcal{T}_1)\sigma(\mathcal{T}_2).$

	We define $\underline{A}'=\{ Y, Y\overline{\in}\mathcal{T}', \mathcal{T}_1=\mathcal{T}^{\prime}_{|Y} \hbox{ and } \mathcal{T}_2=\mathcal{T}^{\prime}_{|X\backslash Y} \}$,
	we notice that $|\underline{A}|=|\underline{A}'|$.
	Then
	\begin{align*}
	 <e_{\mathcal{T}_1}\star e_{\mathcal{T}_2}, \mathcal{T}'>&=<e_{\mathcal{T}_1 \searrow \mathcal{B}\mathcal{T}_2}, \mathcal{B}\mathcal{T}'>\\
  &=\sum_{\underline{v}=(v_1,..., v_n)\in X_2} <e_{\mathcal{T}_1 \searrow_{\underline{v}} \mathcal{B}\mathcal{T}_2}, \mathcal{B}\mathcal{T}'>\\
  &=\sum \limits_{\underset{\scriptstyle\mathcal{B}\mathcal{T}'=\mathcal{T}_{1,1} \searrow_{v_1}(\mathcal{T}_{1,2} \searrow_{v_2}...(\mathcal{T}_{1,n} \searrow_{v_n} \mathcal{B}\mathcal{T}_2)...)}{\underline{v}=(v_1,..., v_n)\in X_2}} \sigma(\mathcal{B}\mathcal{T}')\\
  &=|\underline{B}|\sigma(\mathcal{T}')\\
  &=|\underline{A}'| \sigma(\mathcal{T}_1)\sigma(\mathcal{T}_2)\\
  &=\sum \limits_{\underset{\scriptstyle\mathcal{T}_1=\mathcal{T}^{\prime}_{|Y},\, \mathcal{T}_2=\mathcal{T}^{\prime}_{|X\backslash Y}}{Y\overline{\in}\mathcal{T}'  }} \sigma({\mathcal{T}_1})\sigma({\mathcal{T}_2})\\
  &=\sum_{Y\overline{\in}\mathcal{T}'}<e_{\mathcal{T}_1}, \mathcal{T}^{\prime}_{|Y}><e_{\mathcal{T}_2}, \mathcal{T}^{\prime}_{|X\backslash Y}>\\
  &=<e_{\mathcal{T}_1}\otimes e_{\mathcal{T}_2}, \Delta_{\searrow}(\mathcal{T}')>.
  \end{align*}
  \end{proof}
 \section{Relation between $\searrow$ and $\nearrow$}\label{Relation between}
 In this part we define the law $\nearrow$ on $\mathbb{V}$ by: 
  For all $\mathcal{T}=(X, \leq_{\mathcal{T}})$ and $\mathcal{S}=(Y, \leq_{\mathcal{S}})$ be two finite connected topological spaces,
 \begin{align*}
  \mathcal{T}\nearrow \mathcal{S}:=j\big( j(\mathcal{T})\searrow j(\mathcal{S}) \big),  
\end{align*}
 where $j$ is the involution which transforms $\le$ into $\ge$. In other words, open subsets in $\mathcal{T}$ are closed subsets in $j(\mathcal{T})$ and vice-versa.
 In particular, it is obvious that $ (\mathbb{V},\nearrow) $ is a twisted pre-Lie algebra due to the fact that $(\mathbb{V},\searrow)$ is a twisted pre-Lie algebra.
 \begin{definition}
     For any finite set $A$ and for any pair of parts $A_1$, $A_2$ of $A$ with $A_1\cap A_2=\varnothing$, we define $\Psi_{A_1,\, A_2}:\mathbb T_A\to\mathbb T_A$,
     as follows: for any topology $\mathcal{T}\in \mathbb{T}_A$, $\Psi_{A_1,\, A_2}(\mathcal T)=(A, \le)$, where $\le$ defined by
     \begin{itemize}
    \item If $a\in A_1$, and $b\in A_2$ then $a$ and $b$ are uncomparable,
    \item otherwise, we have $a\le b$ if and only if $a\le_{\mathcal T} b$.
\end{itemize}
 \end{definition}
 \begin{proposition}
  For any finite set $A$ and for any pair of parts $A_1$, $A_2$ of $A$ with $A_1\cap A_2=\varnothing$,  and let $\mathcal{T}\in \mathbb{T}_A$ then\\
 1) $\Psi_{A_1,\, A_2}(\mathcal T)=(A, \le)$ is a finite topological space.\\
 2) $\Psi_{A_1,\, A_2}$ is a projector.
 \end{proposition}
 \begin{proof}
 1) Let $A=A_1\sqcup A_2$, with $A_1\cap A_2=\varnothing$, and let $\mathcal{T}\in \mathbb{T}_A$, we must show that $\leq$ is a preorder relation on $A$:\\
 \textbf{Reflexivity}; Let $x\in A$, then $x\in A_1$ or $x\in A_2$.\\
If $x\in A_1$, we have $x\leq_{\mathcal{T}} x$ then $x\leq x$,
same thing if $x\in A_2$.\\
\textbf{Transitivity}; Let $x, y, z\in A$ such that $x\leq y$ and $y\leq z$. So we have two possible cases:
\begin{itemize}
    \item First case; $x, y, z\in A_1$, and $(x\leq y$ and $y\leq z)$, then $(x\leq_{\mathcal{T}} y \hbox{ and } y\le_{\mathcal{T}} z)$.\\
Since $\leq_{\mathcal{T}}$ is transitive, then $x\leq_{\mathcal{T}} z$, then $x\leq z$.
    \item Second case; $x, y, z\in A_2$, likewise the first case.
\end{itemize}
2) Let $\mathcal{T}\in \mathbb{T}_A$, we must show that $\Psi_{A_1,\, A_2}(\mathcal T)=\Psi^{2}_{A_1,\, A_2}(\mathcal T)$.\\
If not $\mathcal{T'}=(A, \le')=\Psi_{A_1,\, A_2}(\mathcal T)$, then $\Psi^{2}_{A_1,\, A_2}(\mathcal T)=\Psi_{A_1,\, A_2}(\mathcal T')=(A, \le)$, where $\le$ defined by
\begin{itemize}
    \item If $a\in A_1$, and $b\in A_2$ then $a$ and $b$ are uncomparable,
    \item otherwise, we have $a\le b$ if and only if $a\le' b$.
\end{itemize}
And since we have, $a\le' b$ if and only if $a\le_{\mathcal{T}} b$, then $\le$ defined by
\begin{itemize}
    \item If $a\in A_1$, and $b\in A_2$ then $a$ and $b$ are uncomparable,
    \item otherwise, we have $a\le b$ if and only if $a\le_{\mathcal{T}} b$.
\end{itemize}
Then $\Psi_{A_1,\, A_2}=\Psi^{2}_{A_1,\, A_2}$.
 \end{proof}
 \begin{theorem}\label{th.}
 	Let $\mathcal{T}=(X, \leq_{\mathcal{T}})$, $\mathcal{S}=(Y, \leq_{\mathcal{S}})$ and $\mathcal{U}=(Z, \leq_{\mathcal{U}})$ be three finite connected topological spaces, and let $s\in Y$, $u\in Z$. The following diagram is commutative:
	$$
		\xymatrix{
			\mathbb{V}_X \otimes \mathbb{V}_Y \otimes \mathbb{V}_Z \ar[rr]^{id \otimes \searrow_u} \ar[d]_{\nearrow^{s} \otimes id} && \mathbb{V}_X \otimes \mathbb{V}_{Y\sqcup Z} \ar[dd]^{\nearrow^s}\\
			\mathbb{V}_{X\sqcup Y} \otimes \mathbb{V}_Z \ar[d]_{\searrow_u } && \\
			\mathbb{V}_{X\sqcup Y\sqcup Z} \ar[rr]_{\Psi_{X,\, Z}} && \mathbb{V}_{X\sqcup Y\sqcup Z}
		}
			$$
 \end{theorem}
	\begin{proof}
	Let $\mathcal{T}=(X, \leq_{\mathcal{T}})$, $\mathcal{S}=(Y, \leq_{\mathcal{S}})$ and $\mathcal{U}=(Z, \leq_{\mathcal{U}})$ be three finite connected topological spaces, and let $s\in Y$, $u\in Z$, then for $\mathcal{W}=(X\sqcup Y\sqcup Z, \le_{\mathcal{W}})=(\mathcal{T}\nearrow^s \mathcal{S})\searrow_u \mathcal{U}$, we have
	\begin{itemize}
    \item for all $x \in X$, $x\leq_{\mathcal{W}}s$,
    \item and for all $ y \in X\sqcup Y$, $u\leq_{\mathcal{W}}y$,
\end{itemize}
then 
\begin{itemize}
    \item for all $x \in X$, $x\leq_{\Psi_{X, Z}(\mathcal{W})}s$,
    \item and for all $ y \in Y$, $u\leq_{\Psi_{X,\, Z}(\mathcal{W})}y$,
\end{itemize}	
then $\Psi_{X,\, Z}(\mathcal{W})$ is connected.\\
moreover we have
$\mathcal{T}\nearrow^s \mathcal{S}=(X\sqcup Y, \leq^{'}_1)$, with $\leq^{'}_1$ defined on $X\sqcup Y$ by:\\
	$x, y\in X\sqcup Y$ et $x\leq^{'}_1 y$ if and only if:
\begin{itemize}
    \item Either $x, y \in X$ and $x\leq_{\mathcal{T}}y$,
    \item or $x, y \in Y$ and $x\leq_{\mathcal{S}}y$,
    \item or $x\in X$, $y\in Y$ and $s\leq_{\mathcal{S}} y$,
\end{itemize}
then $(\mathcal{T}\nearrow^s \mathcal{S})\searrow_u \mathcal{U}=(X\sqcup Y\sqcup Z, \leq_{\mathcal{W}})$, with $\leq_{\mathcal{W}}$ defined on $X\sqcup Y\sqcup Z$ by:\\
	$x, y\in X\sqcup Y\sqcup Z$ et $x\leq_{\mathcal{W}} y$ if and only if:
	\begin{itemize}
    \item Either $x, y \in X\sqcup Y$ and $x\leq'_1y$,
    \item or $x, y \in Z$ and $x\leq_{\mathcal{U}} y$,
    \item or $x\in Z$, $y\in X\sqcup Y$ and $x\leq_{\mathcal{U}} u$,
\end{itemize}
then 
$x, y\in X\sqcup Y\sqcup Z$ et $x\leq_{\mathcal{W}} y$ if and only if:
	\begin{itemize}
    \item Either $x, y \in X$ and $x\leq_{\mathcal{T}}y$,
    \item or $x, y \in Y$ and $x\leq_{\mathcal{S}}y$,
    \item or $x\in X$, $y\in Y$ and $s\leq_{\mathcal{S}} y$,
    \item or $x, y \in Z$ and $x\leq_{\mathcal{U}} y$,
    \item or $x\in Z$, $y\in Y$ and $x\leq_{\mathcal{U}} u$,
    \item or $x\in Z$, $y\in X$ and $x\leq_{\mathcal{U}} u$,
\end{itemize}
if we apply $\Psi_{X,\, Z}$ to $(\mathcal{T}\nearrow^s \mathcal{S})\searrow_u \mathcal{U}$, we can eliminate the cases: $x\in Z$, $y\in X$ and $x\leq_{\mathcal{U}} u$.\\
On the other hond\\
 $\mathcal{S}\searrow_u \mathcal{U}=(Y\sqcup U, \leq_1)$, with $\leq_1$ defined on $Y\sqcup Z$ by:\\
	$x, y\in Y\sqcup Z$ et $x\leq_1 y$ if and only if:
\begin{itemize}
    \item Either $x, y \in Y$ and $x\leq_{\mathcal{S}}y$,
    \item or $x, y \in Z$ and $x\leq_{\mathcal{U}}y$,
    \item or $x\in Z$, $y\in Y$ and $x\leq_{\mathcal{U}} u$,
\end{itemize}
then $\mathcal{T}\nearrow^s(\mathcal{S}\searrow_u \mathcal{U})=(X\sqcup Y\sqcup Z, \leq)$, with $\leq$ defined on $X\sqcup Y\sqcup Z$ by:\\
	$x, y\in X\sqcup Y\sqcup Z$ et $x\leq y$ if and only if:
	\begin{itemize}
    \item Either $x, y \in X$ and $x\leq_{\mathcal{T}}y$,
    \item or $x, y \in Y\sqcup Z$ and $x\leq_1 y$,
    \item or $x\in X$, $y\in Y\sqcup Z$ and $s\leq_1 y$,
\end{itemize}
then 
$x, y\in X\sqcup Y\sqcup Z$ et $x\leq y$ if and only if:
	\begin{itemize}
    \item Either $x, y \in X$ and $x\leq_{\mathcal{T}}y$,
    \item or $x, y \in Y$ and $x\leq_{\mathcal{S}}y$,
    \item or $x, y \in Z$ and $x\leq_{\mathcal{U}}y$,
    \item or $x\in Z$, $y\in Y$ and $x\leq_{\mathcal{U}} u$,
    \item or $x\in X$, $y\in Y$ and $s\leq_{\mathcal{S}} y$.
\end{itemize}
Then the equality between
$$\mathcal{T}\nearrow^s(\mathcal{S}\searrow_u \mathcal{U})=\Psi_{X,\, Z} \big( (\mathcal{T}\nearrow^s \mathcal{S})\searrow_u \mathcal{U} \big).$$
	\end{proof}
	\begin{corollary}
	Let $\mathcal{T}=(X, \leq_{\mathcal{T}})$, $\mathcal{S}=(Y, \leq_{\mathcal{S}})$ and $\mathcal{U}=(Z, \leq_{\mathcal{U}})$ be three finite connected topological spaces, and let $s\in Y$, $u\in Z$. Then
	$$\mathcal{T}\searrow_s(\mathcal{S}\nearrow^u \mathcal{U})=\Psi_{X,\, Z} \big( (\mathcal{T}\searrow_s \mathcal{S})\nearrow^u \mathcal{U} \big).$$
 i.e, the following diagram is commutative:
	$$		
		\xymatrix{
			\mathbb{V}_X \otimes \mathbb{V}_Y \otimes \mathbb{V}_Z \ar[rr]^{id \otimes \nearrow^u} \ar[d]_{\searrow_{s} \otimes id} && \mathbb{V}_X \otimes \mathbb{V}_{Y\sqcup Z} \ar[dd]^{\searrow_s}\\
			\mathbb{V}_{X\sqcup Y} \otimes \mathbb{V}_Z \ar[d]_{\nearrow^u } && \\
			\mathbb{V}_{X\sqcup Y\sqcup Z} \ar[rr]_{\Psi_{X,\, Z}} && \mathbb{V}_{X\sqcup Y\sqcup Z}
		}
		$$
	\end{corollary}
	\begin{proof}
	 Let $\mathcal{T}=(X, \leq_{\mathcal{T}})$, $\mathcal{S}=(Y, \leq_{\mathcal{S}})$ and $\mathcal{U}=(Z, \leq_{\mathcal{U}})$ be three finite connected topological spaces, and let $s\in Y$, $u\in Z$.\\
	 We notice $\mathcal{T}'=j(\mathcal{T})$, $\mathcal{S}'=j(\mathcal{S})$ and $\mathcal{U}'=j(\mathcal{U})$, according to theorem \ref{th.}, we have: 
	 $$\mathcal{T}'\nearrow^s(\mathcal{S}'\searrow_u \mathcal{U}')=\Psi_{X,\, Z} \big( (\mathcal{T}'\nearrow^s \mathcal{S}')\searrow_u \mathcal{U}' \big),$$
	 then
	  \hspace{2cm}$j\big(\mathcal{T}'\nearrow^s(\mathcal{S}'\searrow_u \mathcal{U}')\big)=j[\Psi_{X,\, Z} \big( (\mathcal{T}'\nearrow^s \mathcal{S}')\searrow_u \mathcal{U}' \big) ]$,\\
	  then
	   \hspace{2cm}$j(\mathcal{T}')\searrow_sj\big((\mathcal{S}'\searrow_u \mathcal{U}')\big)=\Psi_{X,\, Z} [ j\big((\mathcal{T}'\nearrow^s \mathcal{S}')\big)\nearrow^u j(\mathcal{U}') ]$,\\
	   then 
	   \hspace{2cm}$j(\mathcal{T}')\searrow_s\big(j(\mathcal{S}')\nearrow^u j(\mathcal{U}')\big)=\Psi_{X,\, Z} [ \big(j(\mathcal{T}')\searrow_s j(\mathcal{S}')\big)\nearrow^u j(\mathcal{U}') ]$.\\
	   Then
	   $$\mathcal{T}\searrow_s(\mathcal{S}\nearrow^u \mathcal{U})=\Psi_{X,\, Z} \big( (\mathcal{T}\searrow_s \mathcal{S})\nearrow^u \mathcal{U} \big).$$
	\end{proof}
	\begin{proposition}
	Let $\mathcal{T}=(X, \leq_{\mathcal{T}})$, $\mathcal{S}=(Y, \leq_{\mathcal{S}})$ and $\mathcal{U}=(Z, \leq_{\mathcal{U}})$ be three finite connected topological spaces, then
	$$\mathcal{T}\nearrow(\mathcal{S}\searrow \mathcal{U})-\Psi_{X,\, Z} \big( (\mathcal{T}\nearrow \mathcal{S})\searrow \mathcal{U} \big)=\mathcal{S}\searrow(\mathcal{T}\nearrow \mathcal{U})-\Psi_{Y,\, Z} \big( (\mathcal{S}\searrow \mathcal{T})\nearrow \mathcal{U} \big).$$
	\end{proposition}
	\begin{proof}
	 Let $\mathcal{T}=(X, \leq_{\mathcal{T}})$, $\mathcal{S}=(Y, \leq_{\mathcal{S}})$ and $\mathcal{U}=(Z, \leq_{\mathcal{U}})$ be three finite connected topological spaces, then
	 \begin{align*}
	    \mathcal{T}\nearrow(\mathcal{S}\searrow \mathcal{U})-\Psi_{X,\, Z} \big( (\mathcal{T}\nearrow \mathcal{S})\searrow \mathcal{U} \big)&= \sum_{u\in Z,\, s\in Y\sqcup Z}\mathcal{T}\nearrow^s(\mathcal{S}\searrow_u \mathcal{U})-\sum_{u\in Z,\, s\in Y}\Psi_{X,\, Z} \big( (\mathcal{T}\nearrow^s \mathcal{S})\searrow_u \mathcal{U} \big)\\
	    &=\sum_{u,s\in Z}\mathcal{T}\nearrow^s(\mathcal{S}\searrow_u \mathcal{U})+\sum_{u\in Z,\, s\in Y}\big[\mathcal{T}\nearrow^s(\mathcal{S}\searrow_u \mathcal{U})\\
	    &\hspace*{2,5cm}- \Psi_{X,\, Z} \big( (\mathcal{T}\nearrow^s \mathcal{S})\searrow_u \mathcal{U} \big) \big]\\
	    &=\sum_{u,s\in Z}\mathcal{T}\nearrow^s(\mathcal{S}\searrow_u \mathcal{U})\\
	    &=\sum_{u,s\in Z}\mathcal{S}\searrow_u(\mathcal{T}\nearrow^s \mathcal{U})\\
	    &=\sum_{u,s\in Z}\mathcal{S}\searrow_u(\mathcal{T}\nearrow^s \mathcal{U})+\sum_{s\in Z,\, u\in X}\big[\mathcal{S}\searrow_u(\mathcal{T}\nearrow^s \mathcal{U})\\
	    &\hspace*{2,5cm}- \Psi_{Y,\, Z} \big( (\mathcal{S}\searrow_u \mathcal{T})\nearrow^s \mathcal{U} \big)\big]\\
	    &=\sum_{s\in Z,\, u\in X\sqcup Z}\mathcal{S}\searrow_u(\mathcal{T}\nearrow^s \mathcal{U})-\sum_{u\in X,\, s\in Z}\Psi_{Y,\, Z} \big( (\mathcal{S}\searrow_u \mathcal{T})\nearrow^s \mathcal{U} \big)\\
	    &=\mathcal{S}\searrow(\mathcal{T}\nearrow \mathcal{U})-\Psi_{Y,\, Z} \big( (\mathcal{S}\searrow \mathcal{T})\nearrow \mathcal{U} \big).
	 \end{align*}
	\end{proof}

\textbf{Acknowledgements}: I would like to thank my Ph.D. thesis supervisors: Professor Dominique Manchon and Professor Ali Baklouti for their advice, guidance and constant support.\\

	\textbf{Funding}: This work was supported by the University of Sfax, Tunisia.

\end{document}